\documentclass[a4paper,12pt]{article}

\newenvironment{proof}{\noindent {\bf Proof:}}{\hfill $\Box$}

\newtheorem{theorem}{Theorem}
\newtheorem{lemma}{Lemma}
\newtheorem{proposition}{Proposition}
\newtheorem{corollary}{Corollary}

\newtheorem{assumption}{Assumption}
\newtheorem{remark}{Remark}

\usepackage{booktabs}

\usepackage{amsmath}
\usepackage{amssymb}
\usepackage{amsfonts}
\usepackage{graphicx}
\usepackage{ dsfont }

\usepackage[normalem]{ulem}
\usepackage{color}

\usepackage{tikz}
\usepackage{lipsum}
\newcommand{\mycbox}[1]{\tikz{\path[draw=#1,fill=#1] (0,0) rectangle (1cm,1cm);}}

\title{\bf Controller design and value function approximation for nonlinear dynamical systems}

\begin{document}

\author{Milan Korda$^1$, Didier Henrion$^{2,3,4}$, Colin N. Jones$^1$}

\footnotetext[1]{Laboratoire d'Automatique, \'Ecole Polytechnique F\'ed\'erale de Lausanne, Station 9, CH-1015, Lausanne, Switzerland. {\tt \{milan.korda,colin.jones\}@epfl.ch}}
\footnotetext[2]{CNRS; LAAS; 7 avenue du colonel Roche, F-31400 Toulouse; France. {\tt henrion@laas.fr}}
\footnotetext[3]{Universit\'e de Toulouse; LAAS; F-31400 Toulouse; France.}
\footnotetext[4]{Faculty of Electrical Engineering, Czech Technical University in Prague,
Technick\'a 2, CZ-16626 Prague, Czech Republic.}

\date{Draft of \today}

\maketitle

\begin{abstract}
This work considers the infinite-time discounted optimal control problem for continuous time input-affine polynomial dynamical systems subject to polynomial state and box input constraints. We propose a sequence of sum-of-squares (SOS) approximations of this problem obtained by first lifting the original problem into the space of measures with continuous densities and then restricting these densities to polynomials. These approximations are tightenings, rather than relaxations, of the original problem and provide a sequence of rational controllers with value functions associated to these controllers converging (under some technical assumptions) to the value function of the original problem. In addition, we describe a method to obtain polynomial approximations from above and from below to the value function of the extracted rational controllers, and a method to obtain approximations from below to the optimal value function of the original problem, thereby obtaining a sequence of asymptotically optimal rational controllers with explicit estimates of suboptimality. Numerical examples demonstrate the approach.
\end{abstract}


\begin{center}\small
\end{center}

\section{Introduction}
This paper considers the infinite-time discounted optimal control problem for continuous-time input-affine polynomial dynamical systems subject to polynomial state constraints and box input constraints. This problem has a long history in both control and economics literature. Various methods to tackle this problem have been developed, often based on the analysis of the associated Hamilton-Jacobi-Bellman equation.

In this work we take a different approach: We first lift the problem into an infinite-dimensional space of measures with continuous densities where this problem becomes convex; in fact a linear program~(LP). This lifting is a \emph{tightening}, i.e., its optimal value is greater than or equal to the optimal value of the original problem, and under suitable technical conditions the two optimal values coincide. This infinite-dimensional LP is then further tightened by restricting the class of functions to polynomials of a prescribed degree and replacing nonnegativity constraints by sufficient sum-of-squares (SOS) constraints. This leads to a hierarchy of semidefinite programming~(SDP) tightenings of the original problem indexed by the degree of the polynomials. The solutions to the SDPs yield immediately a sequence of \emph{rational} controllers, and we prove that, under suitable technical assumptions, the value functions associated to these controllers converge \emph{from above} to the value function of the original problem.

We also describe how to obtain a sequence of polynomial approximations converging from above and from below to the value function associated to each rational controller. Combined with existing techniques to obtain polynomial under approximations of the value function of the original problem (adapted to our setting), this method can be viewed as a design tool providing a sequence of rational controllers asymptotically optimal in the original problem with explicit estimates of suboptimality in each step.

The idea of lifting a nonlinear problem to an infinite-dimensional space dates back at least to the work of L. C. Young~\cite{young} and subsequent works of Warga~\cite{warga}, Vinter and Lewis~\cite{vinter_lewis}, Rubio~\cite{rubio} and many others, both in deterministic and stochastic settings. These works typically lift the original problem into the space of measures and this lifting is a \emph{relaxation} (i.e., its optimal value is less than or equal to the optimal value of the original problem) and under suitable conditions the two values coincide.

More recently, this infinite-dimensional lifting was utilized numerically by \emph{relaxing} the infinite-dimensional LP into a finite-dimensional SDP~\cite{sicon} or finite-dimensional LP~\cite{quincampoix}. Whereas the LP relaxations are obtained by classical state- and control-space gridding, the SDP relaxations are obtained by optimizing over truncated moment sequences (i.e., involving only finitely many moments) of the measures and imposing conditions \emph{necessary} for these truncated moment sequences to be feasible in the infinite-dimensional lifted problem. These finite-dimensional relaxations provide \emph{lower bounds} on the value function of the optimal control problem and seem to be difficult to use for control design with strong convergence guarantees; a controller extraction from the relaxations is possible although no convergence (e.g.,~\cite{henrion_synthesis_cdc,quincampoix}) or only very weak convergence can be established (e.g.,~\cite{korda_IFAC,majumdar} in the related context of region of attraction approximation).

Contrary to these works, in this approach we tighten the infinite-dimensional LP by optimizing over polynomial densities of the measures and imposing conditions \emph{sufficient} for these densities to be feasible in the infinite-dimensional lifted problem, thereby obtaining upper bounds as opposed to lower bounds. Crucially, to ensure that polynomial densities of arbitrarily low degrees exist for our problem (and therefore the resulting SDP tightenings are feasible), we work with free initial and final measures and set up the cost function and constraints such that this additional freedom does not affect optimality. Importantly, we do \emph{not} assume that the state constraint set is control invariant, a requirement that is often imposed in the existing literature (e.g., \cite{rantzer_duality_cost_density}) but rarely met in practice.

The presented approach bears some similarity with the density approach of~\cite{rantzer_synthesis} for global stabilization later extended to optimal control (in a purely theoretical setting) in~\cite{rantzer_duality_cost_density} and recently generalized to optimal stabilization of a given invariant set in~\cite{vaidya_optimal} (providing both theoretical results and a practical computation method). However, contrary to~\cite{rantzer_synthesis} we consider the problem of optimal control, not stabilization and moreover we work under state constraints. Contrary to~\cite{vaidya_optimal} we work in continuous time, consider a more general problem (optimal control, not optimal stabilization of a given set) and our approach of finite-dimensional approximation is completely different in the sense that it is based purely on convex optimization and it does not rely on state-space discretization. Moreover, and importantly, our approach comes with convergence guarantees.

Finally, let us mention that this work is inspired by~\cite{lasserre_upper}, where a converging sequence of upper bounds on \emph{static} polynomial optimization problems was proposed, as opposed to a converging sequence of lower bounds as originally developed in~\cite{lasserre_lower}.

\section{Preliminaries}
\subsection{Notation}
We use $L(X;Y)$ to denote the space of all Lebesgue measurable functions defined on a set $X \subset \mathbb{R}^n$ and taking values in the set $Y \subset \mathbb{R}^m$.  If the space $Y$ is not specified it is understood to be~$\mathbb{R}$.  The spaces of integrable functions and essentially bounded functions are denoted by $L^1(X;Y)$ and $L^\infty(X;Y)$, respectively. The spaces of continuous respectively $k$-times continuously differentiable functions are denoted by $C(X;Y)$ respectively $C^k(X;Y)$. By a (Borel) measure we understand a countably additive mapping from (Borel) sets to nonnegative real numbers. Integration of a continuous function $v$ with respect to a measure $\mu$ on a set $X$ is denoted by $\int_X v(x)\,d\mu(x)$ or also $\int v\,d\mu$ when the variable and domain of integration are clear from the context. A probability measure is a measure with unit mass (i.e., $\int 1 d\mu = 1$). The support of a measure $\mu$, defined as the smallest closed set whose complement has zero measure, is denoted by $\mathrm{spt}\,\mu$. The ring of all multivariate polynomials in a variable $x$ is denoted by $\mathbb{R}[x]$, the  vector space of all polynomials of degree no more than $d$ is denoted by $\mathbb{R}[x]_d$, and the vector space of $m$-dimensional polynomial vectors is denoted by $\mathbb{R}[x]^m$. The boundary of a set $X$ is denoted by $\partial X$, the interior by $X^\circ$ and the closure by $\bar{X}$. The Euclidean distance of a point $x$ from a set $X$ is denoted by $\mathrm{dist}_{X}(x) $. For a possibly matrix-valued function $f\in C(X;\mathbb{R}^{n\times m})$ we define $\| f \|_{C^0(X)} := \sup_{x\in X}\max_{i,j}|f_{i,j}(x)|$ and for a vector-valued function $g\in C^1(X;\mathbb{R}^n)$ we define $\|g\|_{C^1(X)} := \| g\|_{C^0(X)} + \| \frac{\partial g}{\partial x}\|_{C^0(X)}$, where $\frac{\partial g}{\partial x}$ denotes the Jacobian of $g$. If clear from the context we write $\| \cdot\|_{C^0}$ for $\| \cdot\|_{C^0(X)}$ and similarly for the $C^1$ norm.

\subsection{SOS programming}\label{sec:sosProg}
Crucial to the material presented in the paper is the ability to decide whether a polynomial $p \in \mathbb{R}[x]$ is nonnegative on a set
\[ 
X = \{x\in \mathbb{R}^n\mid g_i(x) \ge 0, \; i = 1,\ldots,n_g\},
\]
with $g_i\in\mathbb{R}[x]$. A \emph{sufficient} condition for $p$ to be nonnegative on $X$ is that it belongs to the truncated quadratic module of degree $d$ associated to $X$,
\[
Q_d(X) := \Big\{  s_0 + \sum_{i=1}^{n_a} g_i(x) s_i(x)   \mid s_0\in \Sigma_{2 \lfloor\frac{d}{2}\rfloor  } , s_i \in \Sigma_{2 \big\lfloor\frac{(d- \mathrm{deg}\, g_i)}{2}\big\rfloor  } \Big\},
\]
where $\Sigma_{2k}$ is the set of all polynomial sum-of-squares (SOS) of degree at most $2k$. Note in particular that $Q_{d+1}(X)\supset Q_d(X)$. If $p \in Q_d(X)$ for some $d \ge 0$ then clearly $p$ is nonnegative on $X$, and the following fundamental result shows that a certain converse result holds.
\begin{proposition}[Putinar~\cite{putinar}]\label{prop:putinar}
Let $N - \|x \|^2 \in Q_d(X)$ for some $d > 0$ and $N\ge 0$ and let $p \in \mathbb{R}[x]$ be strictly positive on $X$.Then $p \in Q_d(X)$ for some $d \ge 0$.
\end{proposition}
Combining with the Stone-Weierstrass Theorem, as an immediate corollary we get:
\begin{corollary}\label{cor:densNonneg}
Let $f \in C(X)$ be nonnegative on $X$ and let $N - \|x \|^2 \in Q_d(X)$ for some $d > 0$ and $N\ge 0$. Then for every $\epsilon \ge 0$ there exists $d\ge 0$ and $p_d \in Q_d(X)$ such that $\| f - p_d \|_{C^0} < \epsilon$.
\end{corollary}
Corollary~\ref{cor:densNonneg} says that polynomials in $Q_d(X)$ are dense (with respect to the $C^0$ norm) in the space of continuous functions nonnegative on $X$ when we let $d$ tend to infinity.

In the rest of the text we use standard algebraic operations on sets. For instance if we write that $p \in gQ_d(X) + h\mathbb{R}[x]_d$, then it means that $p = g q + hr$ with $q \in Q_d(X)$ and $r \in \mathbb{R}[x]_d$.

The inclusion of $p \in Q_d(X)$ for a fixed $d$ is equivalent to the existence of a positive semidefinite matrix $W$ such that $p(x) = b(x)^{\top} W b(x)$, where $b(x)$ is a basis of $\mathbb{R}[x]_{d/2}$, the vector space of polynomials of degree at most $d/2$. Comparing coefficients leads to a set of affine constraints on the coefficients of $p$ and the entries of $W$. Deciding whether $p \in Q_d(X)$ therefore translates to the feasibility of a semidefinite programming problem with the coefficients of $p$ entering affinely. As a result, optimization of a linear function of the coefficients of $p$ subject to the constraint $p \in Q_d(X)$ translates to a semidefinite programming problem~(SDP) and hence to a well-understood and widely studied class of convex optimization problems for which powerful algorithms and off-the-shelf software are available. See, e.g.,~\cite{lasserre} and the references therein for more details.

\section{Problem statement}
We consider the continuous-time input-affine\footnote{Any dynamical system $\dot{x} = f(x,u)$ depending nonlinearly on $u$ can be transformed to the input-affine form by using state inflation $\begin{bmatrix} \dot{x} \\ \dot{u}  \end{bmatrix} = \begin{bmatrix} f(x,u) \\ v  \end{bmatrix}$, where $u$ is now a part of the state and $v$ a new control input; constraints on $v$ then correspond to rate constraints on $u$. Similarly, cost functions depending non-linearly on $u$ in problem~(\ref{opt:main}) can be handled using state inflation in exactly the same fashion.\label{foot:augment}} controlled dynamical system
\begin{equation}\label{eq:sys}
\dot{x}(t) = f(x(t)) + \sum_{i=1}^m f_{u_i}(x(t))u_i(t),
\end{equation}
where $x \in \mathbb{R}^n$ is the state, $u \in \mathbb{R}^m$ is the control input, and the data are polynomial: $f \in \mathbb{R}[x]^n$, 
$f_{u_i} \in \mathbb{R}[x]^n$, $i=1,\ldots,m$. The system is subject to semi-algebraic state and box\footnote{Any box can be affinely transformed to $[0,\bar{u}]$.} input constraints
\begin{subequations}\label{eq:XU}
\begin{align}
&x(t) \in X:= \{ x \in \mathbb{R}^n \mid g_i (x) \ge 0,\;  i=1,\ldots,n_g      \}, \\
&u(t) \in U:= [0,\bar{u}]^m,
\end{align}
\end{subequations}
where $g \in \mathbb{R}[x]^{n_g}$ and $\bar{u}\ge0$. The set $X$ is assumed compact and the polynomials defining $X$ are assumed to be such that
\begin{equation}\label{eq:barg}
\bar{g}(x) := \prod_{i=1}^{n_g} g_i(x) > 0\quad \forall x\in X^\circ.
\end{equation}

Since $X$ is assumed compact, we also assume, without loss of generality, that the inequalities defining the sets $X$ contain the inequality $N-\| x\| ^2 \ge 0$ for some $N \ge 0$.


The goal of the paper is to (approximately) solve the following optimal control problem (OCP):

\begin{equation}\label{opt:main}
\begin{array}{rclll} \displaystyle
V(x_0) & := & \inf\limits_{u(\cdot),\tau(\cdot)} &  \int_0^{\tau(x_0)} e^{-\beta t} [l_x(x(t)) + \sum_{i=1}^ml_{u_i}(x(t))u_i(t) ]\,dt + e^{-\beta \tau}M   \vspace{1.5mm}\\ \displaystyle
&& \hspace{0.6cm} \mathrm{s.t.}  & x(t) = x_0 + \int_0^t f(x(s)) +  \sum_{i=1}^m f_{u_i}(x(s))u_i(s)\,ds, \vspace{1.5mm}\\ \displaystyle
&&& (x(t),u(t)) \in X\times U\;\; \forall t \in [0,\tau(x_0)]  \vspace{1.5mm} \\ \displaystyle
&&& u \in L^{\infty}([0,\tau(x_0)];U),\; \tau \in L(X;[0,\infty])
\end{array}
\end{equation}
where $\beta > 0$ is a given discount factor and $M$ is a constant chosen such that
\begin{equation}\label{eq:M}
M > \beta^{-1}\sup_{x\in X, u \in U}\{ l(x,u)\},
\end{equation}
where the joint stage cost
\begin{equation}\label{eq:joinStageCost}
l(x,u):=\l_x(x) + \sum_{i=1}^m l_{u_i}(x)u_i
\end{equation}
is, without loss of generality, assumed to be nonnegative on $X\times U$. The state and input stage cost functions $l_x$ and $l_{u_i}$, $i=1,\ldots,m$, are assumed to be polynomial. The function $\tau$ in {OCP (\ref{opt:main}) is referred to as a \emph{stopping function}; the optimization is therefore both over the control input $u$ and over the final time $\tau(x_0)$, which can be finite or infinite and can depend on the initial condition $x_0$.

The function $x \mapsto V(x)$ in (\ref{opt:main}) is called the {\it value function}. The reason for choosing the slightly non-standard objective function in~(\ref{opt:main}) is because with this objective function the value function $V$ is bounded (by $M$) on $X$ and it coincides with the standard\footnote{By standard we mean a discounted optimal control problem with cost $\int_0^\infty e^{-\beta t} [l_x(x(t)) + \sum_{i=1}^ml_{u_i}(x(t))u_i(t) ]\,dt $ and no stopping function.} discounted infinite-horizon value function for all initial conditions $x_0 \in X$ for which the trajectories can be kept within the state constraint set $X$ forever using admissible controls, i.e., for all $x_0$ in the maximum control invariant set associated to the dynamics~(\ref{eq:sys}) and the constraints~(\ref{eq:XU}). To see the first claim, set $\tau(x_0) = 0$ for all $x_0 \in X$. To see the second claim notice that with $M$ chosen as in~(\ref{eq:M}), it is always beneficial to continue the time evolution whenever possible and therefore $\tau(x_0) = +\infty$ for all $x_0$ in the maximum controlled invariant set associated to (\ref{eq:sys}) and (\ref{eq:XU}).

\begin{remark}\label{rem:M}
A constant $M$ satisfying~(\ref{eq:M}) can be found either by analytically evaluating the supremum in~(\ref{eq:M}) or by using the techniques of~\cite{lasserre_lower} to find an upper bound.
\end{remark}


Given a Lipschitz continuous \emph{feedback} controller $u \in C(X;U)$ and a stopping function $\tau\in L(X;[0,\infty])$, the ODE~(\ref{eq:sys}) has a unique solution and we let $V_{u,\tau} \in L(X;[0,\infty])$ denote the value function attained by $(u,\tau)$ in~(\ref{opt:main}), i.e., setting $u(t) = u(x(t))$. By $V_u$ we denote the value function $V_{u,\tau_u^\star}$, where $\tau^\star_u \in L(X;[0,\infty])$ is the optimal stopping function associated to~$u$. Note that, by the choice of $M$ in~(\ref{eq:M}), the optimal stopping function $\tau_u^\star$ is equal to the first hitting time of the complement of the constraint set $X$, i.e., \[\tau_u^\star(x_0) = \mathrm{inf}\{ t\ge 0 \mid x(t\! \mid \! x_0)\notin X \},\] where $x(t\! \mid \! x_0)$ is the trajectory of~(\ref{eq:sys}) with $u(t) = u(x(t))$ starting from $x_0$. Notice also that $V_{u,\tau}(x) \ge V(x)$ for all $x\in X$ and that for any pair $(u,\tau)$ feasible in~(\ref{opt:main}) we have $V_{u,\tau}(x) \le M$ for all $x\in X$.

Throughout the paper, we make the following technical assumption:

  \begin{assumption}\label{as:smoothCont}
There exists a sequence of Lipschitz continuous feedback controllers $\{u^k \in C(X;U)  \}_{k=1}^\infty$ and stopping functions $\{\tau^k \in L(X; [0,\infty])\}_{k=1}^\infty$ feasible in~(\ref{opt:main}) such that
\begin{equation}\label{eq:valConv}
\lim_{ k\to\infty } \int_{X} (V_{u^k,\tau^k}(x)-V(x))dx = 0
\end{equation}
and such that for every $k \ge 0$ there exist a function $\rho^k\in C^1(X)$ and a scalar $\gamma^k >0$
such that $\rho^k(x) = 0$ if $\mathrm{dist}_{\partial X}(x) < \gamma_k$ and
\begin{equation}\label{eq:smoothDens2}
\int_X \int_0^{\tau^k(x_0)} e^{-\beta t} v(x^k(t \! \mid \! x_0))\, dt\,dx_0 = \int_X v(x) \rho^k(x)\,dx \quad \forall v \in C(X),
\end{equation}
where $x^k(\cdot \! \mid \! x_0)$ denotes the solution to~(\ref{eq:sys}) controlled by $u^k$.
\end{assumption}

\begin{remark}\label{rem:L1}
Note that $V_{u^k,\tau^k} \ge V$ on $X$ by construction and therefore~(\ref{eq:valConv}) is equivalent to the $L^1$ convergence of $V_{u^k,\tau^k}$ to $V$.
\end{remark}

%

Assumption~\ref{as:smoothCont} says that the optimal control inputs and stopping functions for OCP (\ref{opt:main}) can be well approximated by Lipschitz continuous \emph{feedback} controllers and measurable stopping functions such that the resulting densities of the discounted occupation measures are continuously differentiable and vanish near the boundary of $X$. Note that the existence of an optimal feedback controller, as well as whether it can be well approximated by Lipschitz controllers, are subtle issues. Similarly it is a subtle issue whether asymptotically optimal stopping functions can be found such that the associated densities $\rho^k$ in~(\ref{eq:smoothDens2}) are continuously differentiable and vanish near the boundary of $X$ (note, however, that the left hand side of~(\ref{eq:smoothDens2}) can always be represented as $\int_X v(x) d\mu^k(x)$ for some nonnegative measure~$\mu^k$). This problem is of rather technical nature and has been studied in the literature (e.g.,~\cite[Section~1.4]{crippa_thesis} or~\cite{rantzer_converse}), where affirmative results have been established in related settings. We do not undertake a study of this problem here and rely on Assumption~\ref{as:smoothCont}, which is, for ease of reading, not stated in its most general form. For example, the functions $\rho^k$ do not need to be $C^1$ but only weakly differentiable and the integration on the left-hand side of~(\ref{eq:smoothDens2}) can be weighted by a nonnegative function $\rho^k_0 \in L_1(X)$ satisfying $\rho_0^k \ge 1$ on $X$ and $\rho_0^k \to 1$ in $L_1(X)$. In addition, we conjecture that it is enough to require $\rho^k = 0$ on $\partial X$ and not necessarily on some neighborhood of $\partial X$; this is in particular the case when $X$ is a box or a ball but we expect all the results of the paper to hold with a general semialgebraic set for which the defining functions satisfy~(\ref{eq:barg}).




The main result of this paper is a hierarchy of sum-of-squares (SOS) problems providing an explicit sequence of \emph{rational} feedback controllers $u^k \in C^\infty(X;U)$ such that, under Assumption~\ref{as:smoothCont}, (\ref{eq:valConv}) holds with $\tau^k = \tau_{u^k}^\star$, i.e., a sequence of asymptotically optimal rational controllers in the sense of the $L^1$ convergence of the associated value functions (see Remark~\ref{rem:L1}).



\section{Converging hierarchy of solutions}\label{hierarchy}

In this section we present an infinite-dimensional linear program (LP) in the space of continuous functions whose sum-of-squares (SOS) approximations provide a sequence of \emph{rational} controllers $u^k$ satisfying~(\ref{eq:valConv}). This infinite-dimensional LP is closely related (and in a weak sense equivalent) to OCP (\ref{opt:main}); the rationale behind the derivation of the LP and its relation to OCP (\ref{opt:main}) is detailed in Section~\ref{sec:rationale}.

 The infinite-dimensional LP reads
\begin{equation}\label{opt:dual_inf}
\begin{array}{rclll}
 & \inf\limits_{\rho,\,\rho_0,\; \rho_T,\; {\sigma}} &  \int_X l_x(x)\rho(x) \, dx + \sum_{i=1}^m\int_X l_{u_i}(x)\sigma_i(x) \, dx + M\int_X \rho_T(x) \, dx  \\
& \hspace{0.0cm} \mathrm{s.t.}  & \rho_T  - \rho_0 + \beta\rho + \mathrm{div}(\rho f) + \sum_{i = 1}^m\mathrm{div}(\sigma_i f_{u_i}) = 0\\
&& \rho \le 0  & \hspace{-1.5cm} \mathrm{on} \:\: \partial X\\
&& \rho_0 \ge 1 \:\: & \hspace{-1.5cm} \mathrm{on}\:\: X \\
&& \bar{u}\rho \ge \sigma_i  & \hspace{-1.5cm} \mathrm{on}\:\: X, \;\; i=1,\ldots,m.\\
&& \rho_T \ge 0 & \hspace{-1.5cm} \mathrm{on}\:\: X \\
&& \sigma_i \ge 0  & \hspace{-1.5cm} \mathrm{on}\:\: X, \;\; i=1,\ldots,m.
\end{array}
\end{equation}
The optimization in~(\ref{opt:dual_inf}) is over functions $(\rho,\rho_0, \rho_T, {\sigma})\in C^1(X)\times C(X)\times C(X)\times C^1(X)^{m}$ with ${\sigma} = (\sigma_1,\ldots,\sigma_m)$.

The optimal value of~(\ref{opt:dual_inf}) will be denoted by $p^\star$. The value attained in~(\ref{opt:dual_inf}) by any tuple of densities $(\rho,\rho_0, \rho_T, \sigma) $ feasible in~(\ref{opt:dual_inf}) will be denoted by~$p(\rho,\rho_0, \rho_T, \sigma) $.

\begin{remark}[Non-uniform weighting]\label{rem:nonUnifWeight} Note that we could have imposed $\rho_0 \ge \bar{\rho}_0$ for any polynomial $\bar{\rho}_0 $ nonnegative on $X$. Choosing a different $\bar{\rho}_0$ has no impact on the asymptotic convergence of the value functions established in the rest of the paper as long as $\bar{\rho}_0$ is strictly positive on~$X$. It may, however, influence the speed of convergence in different subsets of~$X$. In general we expect faster convergence where $\bar{\rho}_0$ is large and slower convergence where it is small. Choosing a non-constant $\bar{\rho}_0$ therefore allows to assign a different importance to different subsets of $X$.
\end{remark}


The infinite-dimensional LP~(\ref{opt:dual_inf}) is then approximated by a hierarchy of sum-of-squares~(SOS) problems, which immediately translate to finite-dimensional semidefinite programs~(SDPs).

The SOS approximation of degree $d$ of~(\ref{opt:dual_inf}) reads
\begin{equation}\label{opt:dual_sos}
\begin{array}{rclll}
 & \inf\limits_{(\rho,\rho_0,\rho_T,{\sigma})\in\mathbb{R}[x]_d^{3+m}} &  \int_X l_x(x) \rho(x) \, dx + \sum_{i=1}^m\int_X l_{u_i}(x)\sigma_i(x)\, dx + M\int_X \rho_T(x) \, dx \\
& \hspace{0.0cm} \mathrm{s.t.}  & \rho_T  - \rho_0 + \beta\rho + \mathrm{div}(\rho f) + \sum_{i = 1}^m\mathrm{div}(\sigma_i f_{u_i}) = 0\\
&& -\rho \in Q_d(X) +g_i\mathbb{R}[x]_{d-\mathrm{deg}\, g_i} + \bar{g}\mathbb{R}[x]_{d-\mathrm{deg}\, \bar{g}}  & \hspace{-2cm} i=1,\ldots,n_g\\
&& \rho_0 - 1 \in Q_d(X)  \\
&& \bar{u}\rho - \sigma_i \in Q_d(X) + \bar{g}Q_{d-\mathrm{deg}\, \bar{g}}(X) & \hspace{-2cm} i=1,\ldots,m\\
&& \rho_T \in Q_d(X)\\
&& \sigma_i\in Q_d(X) + \bar{g}Q_{d-\mathrm{deg}\, \bar{g}}(X),  & \hspace{-2cm} i=1,\ldots,m.
\end{array}
\end{equation}
Once a basis for $\mathbb{R}[x]_d$ is fixed (e.g., the standard monomial basis), the objective becomes linear in the coefficients of polynomials $\rho$, $\sigma$ and $\rho_T$,  and the equality constraint is imposed by equating the coefficients. The inclusions in the quadratic modules translate to semidefinite constraints and affine equality constraints; see Section~\ref{sec:sosProg}. Optimization problem~(\ref{opt:dual_sos}) therefore immediately translates to an SDP.


\begin{remark}[Feasibility]\label{rem:feasibility}
Trivially, any feasible solution to~(\ref{opt:dual_sos}) is feasible in~(\ref{opt:dual_inf}). Also, problem~(\ref{opt:dual_sos}) is feasible for any $d\ge 0$. Indeed $(\rho,\rho_0,\rho_T,\sigma) = (0,1,1,0)$ is always feasible in~(\ref{opt:dual_sos}). See also Remark~\ref{rem:termMeasure} below.
\end{remark}

If non-uniform weighting of initial conditions (see Remark~\ref{rem:nonUnifWeight}) was required, the constraint $ \rho_0 - 1 \in Q_d(X)$ would be replaced by  $\rho_0 - \bar{\rho}_0 \in Q_d(X)$ for a polynomial weighting function $\bar{\rho}_0 $ nonnegative on $X$.

Given an optimal solution $(\rho^d,\rho_0^d,\rho_T^d,{\sigma}^d)$ to~(\ref{opt:dual_sos}), we define a rational control law $u^d$ by
\begin{equation}\label{eq:ud_def}
u^d_i(x) := \frac{\sigma_i^d(x)}{\rho^d(x)} \; \; \forall x \in X, \;i=1,\ldots,m.
\end{equation}

The main result of the paper is the following theorem stating that the controllers $u^d$ are asymptotically optimal:
 
\begin{theorem}\label{thm:main}
For all $d \ge 0$ we have $u^d(x) \in U$ for all $x \in X$ and if Assumption~\ref{as:smoothCont} holds, then
\begin{equation}
\lim_{d\to \infty} \int_X (V_{u^d}(x)  - V(x)) \,  dx = 0,
\end{equation}
that is, $V_{u^d} \to V$ in $L_1(X)$ (note that $V_{u^d} \ge V$ on $X$).
\end{theorem} 
 

\section{Rationale behind the LP formulation (\ref{opt:dual_inf}) and proof of the main Theorem~\ref{thm:main}}\label{sec:rationale}
This section explains the rationale behind the LP problem~(\ref{opt:dual_inf}) and its relation to the OCP~(\ref{opt:main}) and gives the proof of Theorem~\ref{thm:main}. First, we lift the original problem~(\ref{opt:main}) into the space of measures with nonnegative densities in $C(X)$; this lifting is problem~(\ref{opt:dual_inf}). Next we tighten the problem by considering only polynomials of prescribed degree and with nonnegativity constraints enforced via SOS conditions; this is problem~(\ref{opt:dual_sos}). Importantly, the lifting~(\ref{opt:dual_inf}) is a \emph{tightening} of the original problem~(\ref{opt:main}) as show in Theorem~\ref{thm:tighten} below. This is in contrast with~\cite{sicon} where the original problem was lifted into the space of measures and this lifting was a \emph{relaxation}.

To be more concrete, observe that any initial measure $\mu_0$, stopping function $\tau \in L(X;[0,\infty])$, and family of trajectories $\{ x(\cdot \!\mid\! x_0)\}_{x_0\in X}$ of~(\ref{eq:sys}) generated by a Lipschitz controller $u\in C(X;U)$ give rise to a triplet of measures defined by
\begin{subequations}\label{eq:meas_def}
\begin{align}
\int_X v(x) d\mu(x) &= \int_{X}   \int_0^{\tau(x_0)} e^{-\beta t}v(x(t\!\mid\!x_0))\,dt \, d\mu_0(x_0), \\
 \int_X v(x) d\mu_T(x) &= \int_{X} e^{-\beta \tau(x_0) }v(x(\tau(x_0)\!\mid\!x_0)) \, d\mu_0(x_0), \\
 \int_X v(x) d\nu_i(x)  &= \int_{X}   \int_0^{\tau(x_0)} e^{-\beta t}v(x(t\!\mid\!x_0))   u_i(x(t\!\mid \! x_0))\,dt\, d\mu_0(x_0).
\end{align}
\end{subequations}
The measure $\mu$ is called \emph{discounted occupation measure}, the measure $\mu_T$ \emph{terminal measure} and the measures $\nu_i$, $i =1,\ldots,m $, \emph{control measures}. These measures satisfy the \emph{discounted Liouville equation}
\begin{equation}\label{eq:discountLiouville_general}
\int_X v \, d\mu_T(x)  =
 \int_X v \, d\mu_0(x) + \int_X ( \nabla v \cdot f - \beta v)\, d\mu(x) + \sum_{i=1}^m\int_X  \nabla v \cdot f_{u_i} \,d\nu_i(x)
\end{equation}
for all $v \in C^1(X)$. This follows by direct computation; see, e.g., \cite{korda_IFAC}. Notice also that $d\nu_i(x) = u_i(x)d\mu(x)$, i.e., $\nu_i$ is absolutely continuous with respect to $\mu$ with Radon-Nikod\'ym derivative equal to $u_i$.

Crucially, the converse statement is also true, although we have to go from stopping functions to stopping measures:
\begin{theorem} [Superposition]\label{thm:meas_rep_crucial}
If measures $\mu$, $\mu_0$, $\mu_T$ and $\nu_i$, $i = 1,\ldots,m$, satisfy~(\ref{eq:discountLiouville_general}) with $\mathrm{spt}\,\mu_0 \subset X$, $\mathrm{spt}\,\mu \subset X$ and $\mathrm{spt}\,\mu_T \subset X$ and $d\nu_i = u_i d\mu$ for some Lipschitz $u \in C(X,U)$, then there exists an ensemble of probability measures (i.e., measures with unit mass) $\{\tau_{x_0}\}_{x_0 \in X }$ and an ensemble of trajectories  $\{ x(\cdot \!\mid\! x_0)\}_{x_0\in X}$
of the system~(\ref{eq:sys}) controlled with $u(t) = u(x(t))$ such that  $x(t\!\mid \! x_0) \in X$ for all $t \in \mathrm{spt}\,\tau_{x_0}$ and
\begin{subequations}\label{eq:meas_traj}
\begin{align}
 \int_X v(x) \,d\mu_0(x) &= \int_{X} v(x(0\!\mid\! x_0)) \,d\mu_0(x_0),\\
 \int_X v(x)\,d\mu(x) &= \int_{X}  \int_0^{\infty} \int_0^{\tau} e^{-\beta t}v(x(t\!\mid\!x_0))\,dt\, d\tau_{x_0}(\tau) \,d\mu_0(x_0), \label{eq:meas_mu} \\
 \int_X v(x)\,d\mu_T(x) &= \int_{X} \int_0^{\infty} e^{-\beta \tau }v(\tau(x_0)) \, d\tau_{x_0}(\tau) \,d\mu_0(x_0), \label{eq:meas_muT}\\
  \int_X v(x)\,d\nu_i(x) &= \int_{X}  \int_0^{\infty} \int_0^{\tau} e^{-\beta t}v(x(t\!\mid\!x_0))   u_i(x(t\!\mid \! x_0))\,dt\, d\tau_{x_0}(\tau) \,d\mu_0(x_0) \label{eq:meas_nu}
\end{align}
\end{subequations}
for all $v \in C^1(X)$.
\end{theorem}
\begin{proof}
See Appendix~B.
\end{proof}

\begin{remark}[Interpretation of Theorem~2] Theorem~2 says that any measures satisfying~(\ref{eq:discountLiouville_general}) are generated by a superposition of the trajectories of the dynamical system $\dot{x} = f(x) + \sum_{i=1}^m f_{u_i}(x)u_i(x)$, where the superposition is over the final time of the trajectories. Note that there is a unique trajectory corresponding to each initial condition (since the vector field $f(x) + \sum_{i=1}^m f_{u_i}(x)u_i(x)$ is Lipschitz) but this unique trajectory can be stopped at multiple times (in fact at a whole continuum of times) allowing for superposition; this superposition is captured by the stopping measures $\{\tau_{x_0}\}_{x_0 \in X }$. For example, if the $\tau_{x_0}$ is a Dirac measure at a given time, then there is no superposition; if $\tau_{x_0}$ has a discrete distribution, then there is a superposition of finitely or countably many overlapping trajectories starting at $x_0$ stopped at different time instances; if $\tau_{x_0}$ has a continuous distribution then there is a superposition of a continuum of overlapping trajectories starting from $x_0$ stopped at different time instances.
\end{remark}

If in addition the measures $\mu_0$, $\mu$, $\mu_T$ satisfying the discounted Liouville equation (\ref{eq:discountLiouville_general}) are absolutely continuous with respect to the Lebesgue measure with densities $\rho_0 \in C(X)$, $\rho \in C^1(X)$, $\rho_T \in C(X)$ such that $\rho = 0$ on $\partial X$, then these densities satisfy
\begin{equation}\label{eq:disc_Liouville_functional}
\rho_T - \rho_0 + \beta\rho +\mathrm{div}(f \rho) + \sum_{i=1}^m \mathrm{div}(f_{u_i}\sigma_i) = 0
\end{equation}
with $\sigma_i = u_i \rho$, $i=1,\ldots,m$. This follows directly by substituting $d\mu_0 = \rho_0dx$, $d\mu = \rho dx$, $d\mu_T = \rho_Tdx$ and $d\nu_i = u_i d\mu = u_i \rho dx = \sigma_i dx$ in~(\ref{eq:discountLiouville_general}) and using integration by parts. Equation~(\ref{eq:disc_Liouville_functional}) holds almost everywhere in $X$ with a Lipschitz controller $u$, since Lipschitz functions are differentiable almost everywhere and the integration by parts formula applies to them, and everywhere with $u \in C^1(X;U)$.

\begin{remark}[Role of the terminal measure] \label{rem:termMeasure} An important feature of the SOS tightenings~(\ref{opt:dual_sos}) is that they are feasible for arbitrarily low degrees (see Remark~\ref{rem:feasibility}), which is crucial from a practical point of view and is not satisfied with other, more obvious, formulations (e.g., those not involving a stopping function in~(\ref{opt:main})); the reason for this is that, in the absence of a terminal measure (i.e., $\rho_T = 0$), the discounted Liouville equation~(\ref{eq:disc_Liouville_functional}) may not have a solution with a polynomial $\rho$ even though $\rho_0$ and the dynamics are polynomial. Indeed, for example with $f = -x$, $f_{u_i} = 0$, $\beta = 1$, $\rho_0 = 1$ on $X = [-1,1]$ and zero elsewhere, the only solution to~(\ref{eq:disc_Liouville_functional}) with $\rho_T = 0$ is $\rho(x) = -\mathrm{ln}(|x|)$.
\end{remark}

Theorem~\ref{thm:meas_rep_crucial} immediately enables us to prove a representation of the cost of problem~(\ref{opt:dual_inf}) in terms of trajectories of~(\ref{eq:sys}).
\begin{lemma}\label{lem:costRep}
If $(\rho,\rho_0,\rho_T,\sigma)$ is feasible in~(\ref{opt:dual_inf}) and $u = \sigma / \rho$, then
\begin{align}
 p(\rho,\rho_0,\rho_T,{\sigma}) &= \int_X \int_0^{\infty} \int_0^{\tau} e^{-\beta t} l_x(x(t \! \mid \! x_0)) dt \, d\tau_{x_0}(\tau) \rho_0(x_0) \,dx_0 \nonumber \\ & \hspace{1cm} + \sum_{i=1}^m\int_X \int_0^{\infty}  \int_0^{\tau} e^{-\beta t} l_{u_i}(x(t \!\mid\! x_0))u_i(x(t \! \mid \! x_0))\, dt \,  d\tau_{x_0}(\tau) \rho_0(x_0)dx_0 \nonumber \\ & \hspace{1cm} + M  \int_X \int_0^{\infty} e^{-\beta\tau}\, d\tau_{x_0}(\tau) \rho_0(x_0) dx_0, \label{eq:costRep} 
 \end{align}
 where $x(\cdot\! \mid\! x_0)$ are trajectories of~(\ref{eq:sys}) controlled by $u(t) = u(x(t))$ and $\tau_{x_0}$ are stopping probability measures with support $\mathrm{spt}\,\tau_{x_0}$ included in $[0,\infty]$. Moreover the state-control trajectories $x(\cdot \! \mid\! x_0)$ and $u(x(\cdot \!\mid\! x_0))$ are feasible in~(\ref{opt:main}) in the sense that $x(t \! \mid\! x_0) \in X$ and $u(t \!\mid\! x_0)\in U$ for all $t\in \mathrm{spt}\,\tau_{x_0}$.
\end{lemma}
\begin{proof}
Let $(\rho,\rho_0,\rho_T,{\sigma})$ be feasible in (\ref{opt:dual_inf}) and let $p(\rho,\rho_0,\rho_T,{\sigma})$ denote the value attained by $(\rho,\rho_0,\rho_T,{\sigma})$ in~(\ref{opt:dual_inf}). The equality constraint of~(\ref{opt:dual_inf}) is exactly~(\ref{eq:disc_Liouville_functional}). Since the constraint of~(\ref{opt:dual_inf}) implies $\rho = 0$ on $\partial X$, equation (\ref{eq:discountLiouville_general}) holds with $d\mu_0 = \rho_0dx$, $d\mu = \rho dx$, $d\mu_T = \rho_Tdx$ and $d\nu_i = u_i d\mu = u_i \rho dx = \sigma_i dx$, where $u_i = \frac{\sigma_i}{\rho} \in C^1(X;U)$, $i=1,\ldots,m$. By Theorem~\ref{thm:meas_rep_crucial} (setting $v(x) = l_x(x)$ in~(\ref{eq:meas_mu}), $v(x) = 1$ in~(\ref{eq:meas_muT}) and $v(x) = l_{u_i}(x)$ in (\ref{eq:meas_nu})) we obtain the result (noticing that the constraints of~(\ref{opt:dual_inf}) imply that $u(x) \in U$ for all $x\in X$).
\end{proof}
\begin{corollary}\label{cor:costUB}
If $(\rho,\rho_0,\rho_T,\sigma)$ is feasible in~(\ref{opt:dual_inf}) and $u = \sigma / \rho$, then
\begin{equation}\label{eq:costUB}
p(\rho,\rho_0,\rho_T,{\sigma}) \ge \int_X V_u(x_0)\rho_0(x_0)dx_0.
\end{equation}
If in addition the stopping measures $\{\tau_{x_0}\}_{x_0 \in X}$ in~(\ref{eq:costRep}) are equal to the Dirac measures $\{\delta_{\tau(x_0)}\}_{x_0 \in X}$ for some stopping function $\tau \in L(X;[0,\infty])$, then
\begin{equation}\label{eq:costExact}
p(\rho,\rho_0,\rho_T,{\sigma}) = \int_X V_{u,\tau}(x_0)\rho_0(x_0)dx_0.
\end{equation}
\end{corollary}
\begin{proof}
Let $(\rho,\rho_0,\rho_T,{\sigma})$ be feasible in (\ref{opt:dual_inf}). Using Lemma~\ref{lem:costRep}, $p(\rho,\rho_0,\rho_T,{\sigma})$ has representation~(\ref{eq:costRep}), where the state-control trajectories in~(\ref{eq:costRep}) are feasible in~(\ref{opt:main}). Since the measures $\tau_{x_0}$ in~(\ref{eq:costRep}) have unit mass for all $x_0\in X$, we obtain~(\ref{eq:costUB}). If $\tau_{x_0} = \delta_{\tau(x_0)}$ for some stopping function $\tau \in L(X;[0,\infty])$, then the integrals with respect to~$\tau_{x_0}$ in~(\ref{eq:costRep}) become evaluations at $\tau(x_0)$ and hence (\ref{eq:costExact}) holds.
\end{proof}

Corollary~\ref{cor:costUB} immediately implies that the problem~(\ref{opt:dual_inf}) (and hence problem~(\ref{opt:dual_sos})) is a tightening of the original problem~(\ref{opt:main}):

\begin{theorem}\label{thm:tighten}
The optimal value of~(\ref{opt:dual_inf}) of $p^\star$ satisfies
\begin{equation}\label{eq:ineq}
p^* \ge\int_X V(x)\,dx.
\end{equation} 
\end{theorem}
\begin{proof}
Follows from Corollary~\ref{cor:costUB} since $\rho_0 \ge 1$ and $V_u \ge V \ge 0$.
\end{proof}

Now we are in a position to prove the following crucial lemma linking problems (\ref{opt:main}) and (\ref{opt:dual_inf}).
\begin{lemma}\label{lem:conv}
If $\{u^k \in C(X;U)\}_{k=1}^{\infty}$ and $\{\tau^k\in L(X;[0,\infty])\}_{k=1}^{\infty}$ are respectively sequences of controllers and stopping functions satisfying the conditions of Assumption~\ref{as:smoothCont}, then the corresponding densities $\{\rho^k,\rho_0^k,\rho_T^k,{\sigma}^k\}_{k = 1}^\infty$ with $\rho_0^k = 1$ are feasible in~(\ref{opt:dual_inf}) and satisfy
\begin{equation}\label{eq:optEq}
\lim_{k\to \infty }p(\rho^k,\rho_0^k,\rho_T^k,{\sigma}^k) =  \int_X V(x_0)dx_0.
\end{equation}

Conversely, if $\{\rho^k,\rho_0^k,\rho_T^k,{\sigma}^k\}_{k = 1}^\infty$ is a sequence such that $\lim_{k\to\infty} p(\rho^k,\rho_0^k,\rho_T^k,{\sigma}^k) =p^\star$ and if Assumption~(\ref{as:smoothCont}) holds, then equation~(\ref{eq:valConv}) holds with $u^k= {\sigma}^k / \rho^k$.
\end{lemma}
\begin{proof}
To prove the first part of the statement consider the controllers $u^k$, stopping functions $\tau^k$ and densities $\rho^k$ from Assumption~(\ref{as:smoothCont}). Setting $\rho_0^k = 1$ and defining $\sigma_i^k := u_i^k \rho^k$ and $\rho_T^k : = \rho_0^k - \beta\rho^k -\mathrm{div}(\rho^k f) - \sum_{i=1}^m \mathrm{div}(f_{u_i}\sigma_i^k) = 0 $ we see that $(\rho^k,\rho_0^k,\rho_T^k,\sigma^k)$ satisfy~(\ref{eq:disc_Liouville_functional}) with $\rho^k = 0$ on $\partial X$. Therefore~$(\rho^k,\rho_0^k,\rho_T^k,\sigma^k)$ are feasible in~(\ref{opt:dual_inf}). In addition, in view of~(\ref{eq:smoothDens2}), the representation~(\ref{eq:costRep}) holds with $\tau_{x_0} = \delta_{\tau^k(x_0)}$. Therefore by Lemma~\ref{lem:costRep}
\[
 p(\rho^k,\rho_0^k,\rho_T^k,{\sigma}^k) = \int_X V_{u^k,\tau^k}(x_0)\rho^k_0(x_0)dx_0
 \]
and hence~(\ref{eq:optEq}) holds since $\{V_{u^k,\tau^k}\}_{k=1}^\infty$ satisfies~(\ref{eq:valConv}) and $\rho_0^k = 1$ for all $k \ge 0$.

To prove the second part of the statement, let $\{\rho^k,\rho_0^k,\rho_T^k,{\sigma}^k\}_{k=1}^\infty$ be any sequence such that $\lim_{k\to\infty} p(\rho^k,\rho_0^k,\rho_T^k,{\sigma}^k) =p^\star$. Then this sequence satisfies~(\ref{eq:optEq}) by Theorem~\ref{thm:tighten} and by the first part of Lemma~\ref{lem:conv} just proven. Therefore~(\ref{eq:valConv}) holds with $u^k : = \sigma^k/\rho^k$ since
\[
p(\rho^k,\rho_0^k,\rho_T^k,{\sigma}^k) \ge \int_X V_{u^k}(x_0)dx_0
\]
by Corollary~\ref{cor:costUB}.
\end{proof}

We will also need the following result showing that nonnegative $C^1$ functions vanishing on a neighborhood of $\partial X$ can be approximated by polynomials in $Q_d(K)$ vanishing on $\partial X$.
\begin{lemma}\label{lem:polyApprox}
Let $\rho \in C^1(X)$ such that $\rho\geq 0$ on $X$ and $\rho=0$ on $\{x \in X \: : \: \mathrm{dist}_{\partial X}(x) < \zeta\}$ for some $\zeta > 0$. Then for any $\epsilon > 0$ there exists $d \ge 0$ and a polynomial $p_d \in \bar{g}Q_{d-\mathrm{deg}\,\bar{g}}(X)$ such that
\begin{equation*}
\| \rho - p_d \|_{C^1(X)} < \epsilon
\end{equation*}
and $p_d = 0$ on $\partial X$.
\end{lemma}
\begin{proof}
Since $\bar{g}>0$ on $X^\circ$, we can factor $\rho=\bar{g}h$ with $h \in C^1(X)$ given by
\[
h(x) := \begin{cases} \rho(x) / \bar{g}(x) & \mathrm{if}\:\:\mathrm{dist}_{\partial X}(x) \ge \zeta \\
									0 & \mathrm{otherwise}.	 \end{cases}
\]
Since polynomials are dense in $C^1$ there exists for every $\delta > 0$ a polynomial $\hat{h} > 0$ such that 
\begin{equation}
 \| \hat{h} - h \|_{C^1(X)} < \delta.
\end{equation}
Applying Proposition~\ref{prop:putinar} to $\hat{h} $ we see that there exists $\hat{p}_{\hat{d}} \in Q_{\hat{d}}(X)$ for some $\hat{d} \ge 0$ such that
\begin{equation}
\| \hat{h} - \hat{p}_{\hat{d}}\|_{C^1(X)} < \delta.
\end{equation}
Defining $p_d := \hat{p}_{\hat{d}} \bar{g}$ we see that $p_d \in \bar{g}Q_{d-\mathrm{deg}\,\bar{g}}(X)$ with $d = \hat{d} + \mathrm{deg}(\bar{g})$ and that $p_d=0$ on $\partial X$. Finally,
\[
\| \rho - p_d \|_{C^0} = \|h\bar{g} - \hat{p}_{\hat{d}} \bar{g} \|_{C^0} \le \|\bar{g}\|_{C^0} \|h - \hat{p}_{\hat{d}}  \|_{C^0} < 2\delta \|\bar{g}\|_{C^0} 
\]
and
\begin{align*}
\| {\nabla}\rho - {\nabla} p_d\|_{C^0} &= \| \bar{g}\:{\nabla} h+ h\:{\nabla}\bar{g}  - \bar{g}\:{\nabla}\hat{p}_{\hat{d}}+ \hat{p}_{\hat{d}}\:{\nabla}\bar{g}\|_{C^0}  \\
 &\le  \|{\nabla}\bar{g} \|_{C^0} \|h - \hat{p}_{\hat{d}}\|_{C^0} + \|\bar{g}\|_{C^0} \|{\nabla} h - {\nabla}\hat{p}_{\hat{d}} \|_{C^0}  \\
 & \le 2\delta \big(   \|{\nabla}\bar{g} \|_{C^0} +     \| \bar{g}\|_{C^0} \big).
\end{align*}
Therefore choosing $\delta$ such that $2\delta \big(   \|{\nabla}\bar{g} \|_{C^0} +     2\| \bar{g}\|_{C^0} \big) < \epsilon$ gives the desired result.
\end{proof}

Now we are ready to prove our main result, Theorem~\ref{thm:main}.

\textbf{Proof} (of Theorem~\ref{thm:main}):
Consider the sequences $\{u^k \in C^1(X;U)  \}_{k=1}^\infty$, $\{\tau^k \in L(X;[0,\infty])  \}_{k=1}^\infty$ from~Assumption~\ref{as:smoothCont}. By the first part of Lemma~\ref{lem:conv} the sequence of associated densities $(\rho^k,\rho_0^k,\rho_T^k,{\sigma}^k)$ generated by $(u^k,\tau^k)$ is feasible in~(\ref{opt:dual_inf}) and satisfies~(\ref{eq:optEq}). By Assumption~\ref{as:smoothCont}, $\rho^k = 0$ and $ \sigma ^k = 0$ on $\{x\in X\: :\: \mathrm{dist}_{\partial X}(x)< \gamma^k  \} $ (since $ \sigma ^k = u^k\rho^k$) with $\gamma^k > 0$.

 Hence by Lemma~\ref{lem:polyApprox} there exist polynomial densities $\rho^{k,\mathrm{pol}}\in \bar{g}Q_{d_k-\mathrm{deg}\,\bar{g}}(X)$, $\sigma^{k,\mathrm{pol}}\in \bar{g}Q_{d_k-\mathrm{deg}\,\bar{g}}(X)^m$ for some degrees $d^k \ge 0$ such that
\begin{equation}\label{eq:auxFirst}
\| \rho^k- \rho^{k,\mathrm{pol}}\|_{C^1(X)}  < 1/k
\end{equation}
\begin{equation}
\| \sigma_i^k - \sigma_i^{k,\mathrm{pol}} \|_{C^1(X)} < 1/k
\end{equation}

 $\bar{u}\rho^{k,\mathrm{pol}} - \sigma_i^{k,\mathrm{pol}} \in \bar{g}Q_{d^k-\mathrm{deg}\,\bar{g}}(X)$ for all $i=1,\ldots,m$ (since $u^k(x)\in U = [0,\bar{u}]^m $ for all $x\in X$ and hence $\bar{u}\rho^k \ge \sigma_i^k$ on $X$). Notice also that since $\rho^{k,\mathrm{pol}}\in \bar{g}Q_{d^k-\mathrm{deg}\,\bar{g}}(X)  $, we have $-\rho^{k,\mathrm{pol}} \in \bar{g}\mathbb{R}_{d^k-\mathrm{deg}\,\bar{g}} $. Next, since $\rho_0^k \ge 1$ and $\rho_T^k \ge 0$, we can find, by Corollary~\ref{cor:densNonneg}, polynomial densities $\hat{\rho}_0^{k,\mathrm{pol}} \in 1 + Q_{d^k}(X)$ and $\hat{\rho}_T^{k,\mathrm{pol}} \in Q_{d^k}(X)$ such that
 \begin{equation}
\| \rho_0^k - \hat{\rho}_0^{k,\mathrm{pol}} \|_{C^0(X)} < 1/k,
\end{equation}
 \begin{equation}\label{eq:auxLast}
\| \rho_T^k- \hat{\rho}_T^{k,\mathrm{pol}}\|_{C^0(X)}< 1/k.
\end{equation}
Since $(\rho^k,\rho_0^k,\rho_T^k,{\sigma}^k)$ satisfy the equality constraint of~(\ref{opt:dual_inf}) we have
\[
\hat{\rho}_T^{k,\mathrm{pol}} + \beta\rho^{k,\mathrm{pol}} - \hat{\rho}_0^{k,\mathrm{pol}} + \mathrm{div}(\rho^{k,\mathrm{pol}}f) + \sum_{i = 1}^m\mathrm{div}(\sigma_i^{k,\mathrm{pol}}f_{u_i})  = \omega^k
\]
where
\[
\omega^k := \hat{\rho}_T^{k,\mathrm{pol}}-\rho_T^k + \beta(\rho^{k,\mathrm{pol}}-\rho^k) - (\hat{\rho}_0^{k,\mathrm{pol}}-\rho_0^k) + \mathrm{div}[(\rho^{k,\mathrm{pol}}-\rho^k )f] + \sum_{i = 1}^m\mathrm{div}[(\sigma_i^{k,\mathrm{pol}}-\sigma_i^k)f_{u_i}]
\]
is a polynomial such that $\| \omega^k\|_{C^0} \to 0$ as $k\to \infty$ in view of~(\ref{eq:auxFirst})-(\ref{eq:auxLast}). Defining the constants $\epsilon^k =  1/k + \|\omega^k\|_{C^0}$ and setting
\[
\rho_T^{k,\mathrm{pol}} := \hat{\rho}_T^{k,\mathrm{pol}} + \epsilon^k
\]
\[
\rho_0^{k,\mathrm{pol}} := \hat{\rho}_0^{k,\mathrm{pol}} + \epsilon^k + \omega^k
\]
we see that
\[
\rho_T^{k,\mathrm{pol}} + \beta\rho^{k,\mathrm{pol}} - \rho_0^{k,\mathrm{pol}} + \mathrm{div}(\rho^{k,\mathrm{pol}}f) + \sum_{i = 1}^m\mathrm{div}(\sigma_i^{k,\mathrm{pol}}f_{u_i})  = 0,
\]
and $\rho_0^{k,\mathrm{pol}} -1 $ and $\rho_T^{k,\mathrm{pol}}$ are strictly positive on $X$ and hence belong to $Q_{d_k}(X)$. The densities $(\rho^{k,\mathrm{pol}},\rho_0^{k,\mathrm{pol}},\rho_T^{k,\mathrm{pol}},\sigma^{k,\mathrm{pol}})$ are therefore feasible in~(\ref{opt:dual_sos}) for some $d^k \ge 0$. In addition, by construction, $\| \rho_0^{k,\mathrm{pol}} - \rho_0^k \|_{C^0} \to 0$ and $\| \rho_T^{k,\mathrm{pol}} - \rho_T^k \|_{C^0} \to 0$ as $k\to \infty$. Therefore we have obtained a sequence of polynomial densities $(\rho^{k,\mathrm{pol}},\rho_0^{k,\mathrm{pol}},\rho_T^{k,\mathrm{pol}},\sigma^{k,\mathrm{pol}})$ that are feasible in~(\ref{opt:dual_sos}) and such that
\[
\| \rho_0^{k,\mathrm{pol}} - \rho_0^k \|_{C^0} \to 0,\quad \| \rho_T^{k,\mathrm{pol}} - \rho_T^k \|_{C^0} \to 0,\quad \| \rho^{k,\mathrm{pol}} - \rho^k \|_{C^1} \to 0,\quad \| \sigma^{k,\mathrm{pol}} - \sigma^k \|_{C^1} \to 0
\]
as $k \to \infty $. This implies that
\[
|p(\rho^{k,\mathrm{pol}},\rho_0^{k,\mathrm{pol}},\rho_T^{k,\mathrm{pol}},\sigma^{k,\mathrm{pol}}) - p(\rho^k,\rho_0^k,\rho_T^k,{\sigma}^k)| \to 0
\]
and hence $(\rho^{k,\mathrm{pol}},\rho_0^{k,\mathrm{pol}},\rho_T^{k,\mathrm{pol}},\sigma^{k,\mathrm{pol}})$ satisfies~(\ref{eq:optEq}) and so $p(\rho^{k,\mathrm{pol}},\rho_0^{k,\mathrm{pol}},\rho_T^{k,\mathrm{pol}},\sigma^{k,\mathrm{pol}}) \to p^\star$ by Theorem~\ref{thm:tighten}. Therefore~(\ref{eq:valConv}) holds with the rational controllers $u^{k} : = \sigma^{k,\mathrm{pol}} / \rho^{k,\mathrm{pol}}$ by the second part of Lemma~\ref{lem:conv}. This finishes the proof. \hfill $\square$
%
%

\section{Value function approximations}\label{sec:valApprox}

In this section we propose a converging hierarchy of approximations from below and from above to the value function $V_u$ associated to a  rational controller $u = \sigma / \rho$ with $\sigma \in \mathbb{R}[x]^m$ and $\rho \in \mathbb{R}[x]$ satisfying $0\le\sigma_i\le\bar{u}\rho$ on $X$. In addition we describe a hierarchy of approximations from below to the optimal value function~$V$. This is useful as a post-processing step, once a rational control law has been computed as described in Section \ref{hierarchy}}, providing an explicit bound on the suboptimality of the controller.

Note that, trivially, approximations from above to $V_u$ provide approximations from above to~$V$. Defining $\hat{f} = \rho f + \sum_{i=1}^m f_{u_i} \sigma_i \in\mathbb{R}[x]^n$ and $\hat{l} = \rho l_x + \sum_{i=1}^m l_{u_i}\sigma_i \in\mathbb{R}[x]$, the degree $d$ polynomial upper and lower bounds are given by
\begin{equation}\label{opt:val_u_ub}
\begin{array}{rclll}
 & \min\limits_{ \overline{V_u} \in \mathbb{R}[x]_d} &  \int_X \overline{V_u}(x)\,dx  \\
& \hspace{0.0cm} \mathrm{s.t.}  & \beta \rho\overline{V_u}  - \nabla \overline{V_u}   \cdot \hat{f} - \hat{l} \in Q_d(X)  \vspace{1mm}\\
&& \overline{V_u}  - M \in Q_d(X) + \bar{g}\mathbb{R}_{d-\mathrm{deg}\,\bar{g}} ,
\end{array}
\end{equation}
and
\begin{equation}\label{opt:val_u_lb}
\begin{array}{rclll}
 & \max\limits_{\underline{V_u} \in \mathbb{R}[x]_d} & \int_X \underline{V_u} (x)\,dx  \\
& \hspace{0.0cm} \mathrm{s.t.}  & -(\beta \rho\underline{V_u} - \nabla \underline{V_u}   \cdot \hat{f} - \hat{l}) \in Q_d(X) \vspace{1mm} \\
&& M - \underline{V_u}  \in Q_d(X) + \bar{g}\mathbb{R}_{d-\mathrm{deg}\,\bar{g}} ,
\end{array}
\end{equation}
respectively.
Fixing a basis of $\mathbb{R}[x]_d$, the objective functions of (\ref{opt:val_u_ub}) and (\ref{opt:val_u_lb}) become linear in the coefficients of $\overline{V_u}$ respectively $\underline{V_u}$ in this basis. Problems~(\ref{opt:val_u_ub}) and (\ref{opt:val_u_lb}) are therefore convex SOS problems and immediately translate to SDPs (see Section~\ref{sec:sosProg}).

\begin{theorem}\label{thm:Vu_ub_lb}
Let $\overline{V_u}^d$ and $\underline{V_u}^d$ denote solutions to~(\ref{opt:val_u_ub}) and (\ref{opt:val_u_lb}) of degree $d$. Then $\overline{V_u}^d \ge V_u \ge \underline{V_u}^d$ on $X$ and
\begin{equation}\label{eq:convUbLb}
\lim_{d\to \infty }\int_X\overline{V_u}^d(x)\,dx = \int_X V_u(x) \, dx = \lim_{d\to \infty }\int_X\underline{V_u}^d(x)\,dx .
\end{equation}
\end{theorem}
\begin{proof}
See Appendix~A.
\end{proof}

As a simple corollary we obtain a converging sequence of polynomial over-approximations to $V$, the optimal value function of~(\ref{opt:main}):
\begin{theorem}
Let $\overline{V}_{u^{d_1}}^{d_2}$ denote the degree $d_2$ polynomial approximation from above to the value function associated to the rational controller $u^{d_1}$ obtained from~(\ref{opt:dual_sos}) using~(\ref{eq:ud_def}). Then $\overline{V}_{u^{d_1}}^{d_2} \ge V$ on $X$ and
\[
\lim_{d_1 \to \infty}\lim_{d_2 \to \infty} \int_X (\overline{V}_{u^{d_1}}^{d_2}(x) - V(x))\, dx = 0.
\]
\end{theorem}

Now we describe a hierarchy of lower bounds on $V$:
\begin{equation}\label{opt:val_opt_lb}
\begin{array}{rclll}
 & \max\limits_{\underline{V} \in \mathbb{R}[x]_d,\;  p \in \mathbb{R}[x]_d^m} &  \int_X \underline{V} (x)\,dx  \\
& \hspace{0.0cm} \mathrm{s.t.}  & l_x - \beta \underline{V} + \nabla \underline{V} \cdot f + \bar{u} \sum_{i=1}^m p_i \in Q_d(X) \vspace{1mm} \\
&&l_{u_i} + \nabla \underline{V} \cdot f_{u_i} - p_i \in Q_d(X)\\
&& - p_i \in Q_d(X)\\
&& M - \underline{V}  \in Q_d(X) + \bar{g}\mathbb{R}_{d-\mathrm{deg}\,\bar{g}}.
\end{array}
\end{equation}
\begin{theorem}
If $\underline{V} \in \mathbb{R}[x]_d$ is feasible in~(\ref{opt:val_opt_lb}), then $\underline{V} \le V$ on $X$.
\end{theorem}
\begin{proof}
Follows by similar arguments based on Gronwall's Lemma as in the proof of Theorem~\ref{thm:Vu_ub_lb}.
\end{proof}

\begin{remark}
The question whether $\underline{V}$ converges from below to $V$ as degree $d$ in~(\ref{opt:val_opt_lb}) tends to infinity is open (although likely to hold). A proof would require an extension of the superposition Theorem~\ref{thm:meas_rep_crucial_thm} to non-Lipschitz vector fields (in the spirit of the finite-time superposition result of~\cite[Theorem 4.4]{ambrosio_crippa}) or an extension of the argument of~\cite{quincampoix} to the case of $\mu_T \neq 0 $, either of which is beyond the scope of this paper.
\end{remark}

\begin{remark}
Besides closed-loop cost function with respect to the OCP~(\ref{opt:main}), one can assess other aspects of the closed-loop behavior of the dynamical system~(\ref{eq:sys}) controlled by the rational controller $u = { \sigma} / \rho$. In particular, regions of attraction or maximum controlled invariant sets can be estimated by methods of~\cite{roa,inner_controlled_extended,mci_outer}, which extend readily to the case of rational systems.
\end{remark}

\section{Numerical examples}
This section demonstrates the approach on numerical examples. To improve the numerical conditioning of the SDPs solved, we use the Chebyshev basis to parametrize all polynomials. More specifically, we use tensor products of univariate Chebyshev polynomials of the first kind to obtain a multivariate Chebyshev basis. We note, however, that similar results, albeit slightly less accurate could be obtained with the standard multivariate monomial basis (in which case the SDPs can be readily formulated using high level modelling tools such as Yalmip~\cite{yalmip} or SOSOPT~\cite{sosopt}). The resulting SDPs were solved using MOSEK.

\subsection{Nonlinear double integrator}
As our first example we consider the nonlinear double integrator
\begin{align*}
\dot{x}_1 &= x_2 + 0.1x_1^3\\
\dot{x}_2 &= 0.3u
\end{align*}
subject to the constraints $u\in [-1,1]$ and $x\in X := \{x : \|x\|_2 < 1\}$ and stage costs $l_x(x) = x^\top x$ and $l_u(x) = 0$. The discount factor $\beta$ was set to 1; the constant $M$ to $1.01 > \sup_{x \in X}\{ x^\top x\} /\beta = 1$. First we obtain a rational controller of degree six by solving~(\ref{opt:dual_inf}) with $d = 6$. The graph of the controller is shown in Figure~\ref{fig:NL_doubleInteg_Cont}. Next we obtain a polynomial upper bound $\overline{V_u}$ of degree 14 on the value function associated to $u$ by solving (\ref{opt:val_u_ub}) with $d = 14$. To assess suboptimality of the controller $u$ we compare it with a lower bound $\underline{V}$ on the optimal value function of the problem~(\ref{opt:main}) obtained by solving~(\ref{opt:val_opt_lb}) with $d = 14$. The graphs of the two value functions are plotted in Figure~\ref{fig:NL_doubleInteg_valFun}. We see that the gap between the upper bound on $V_u$ and lower bound on $V$ is relatively small, verifying a good performance of the extracted controller. Quantitatively, the average performance gap defined as $100\int_X (\overline{V_u} - \underline{V}) dx / \int_X \underline{V} dx$ is equal to $19.5 \%$.

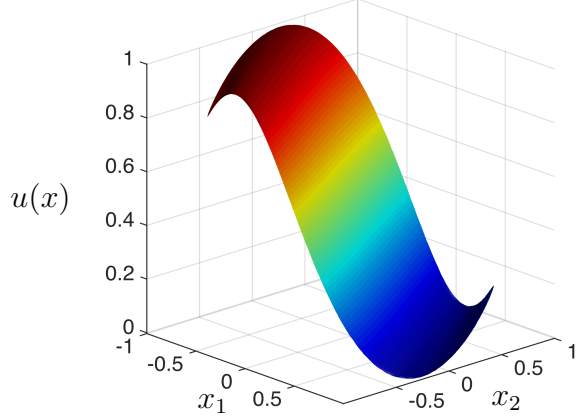
\begin{figure*}[h!]
	\begin{picture}(140,210)
	
	\put(120,0){\includegraphics[width=80mm]{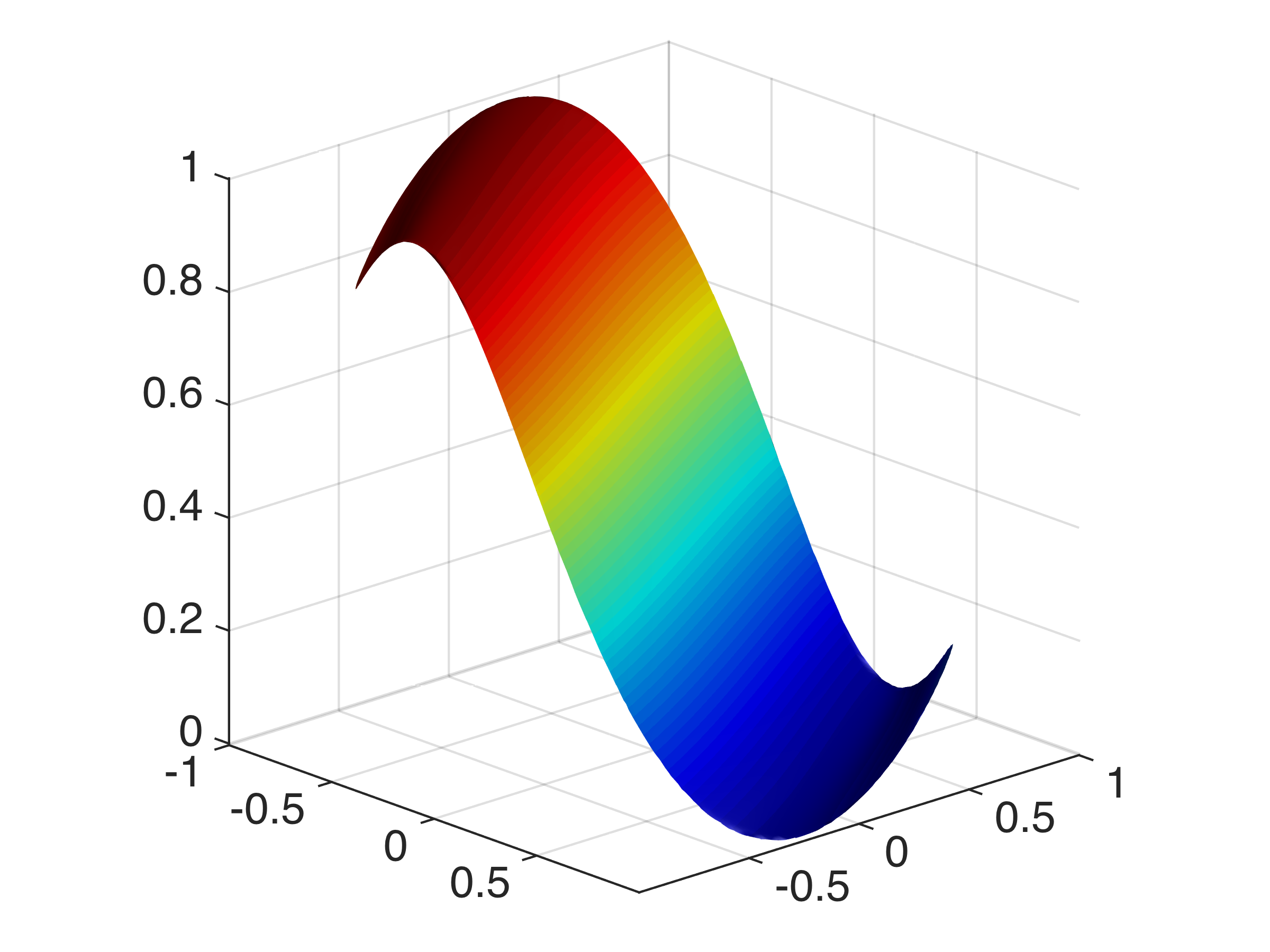}}
	
	\put(180,09){$x_1$}
	\put(290,10){$x_2$}
	\put(110,85){$u(x)$}
	\put(240,-20){\mycbox{white}}

	\end{picture}
	\caption{\small Nonlinear double integrator -- rational controller of degree six.}
	\label{fig:NL_doubleInteg_Cont}
\end{figure*}

\begin{figure*}[h!]
	\begin{picture}(140,210)
	
	\put(120,00){\includegraphics[width=90mm]{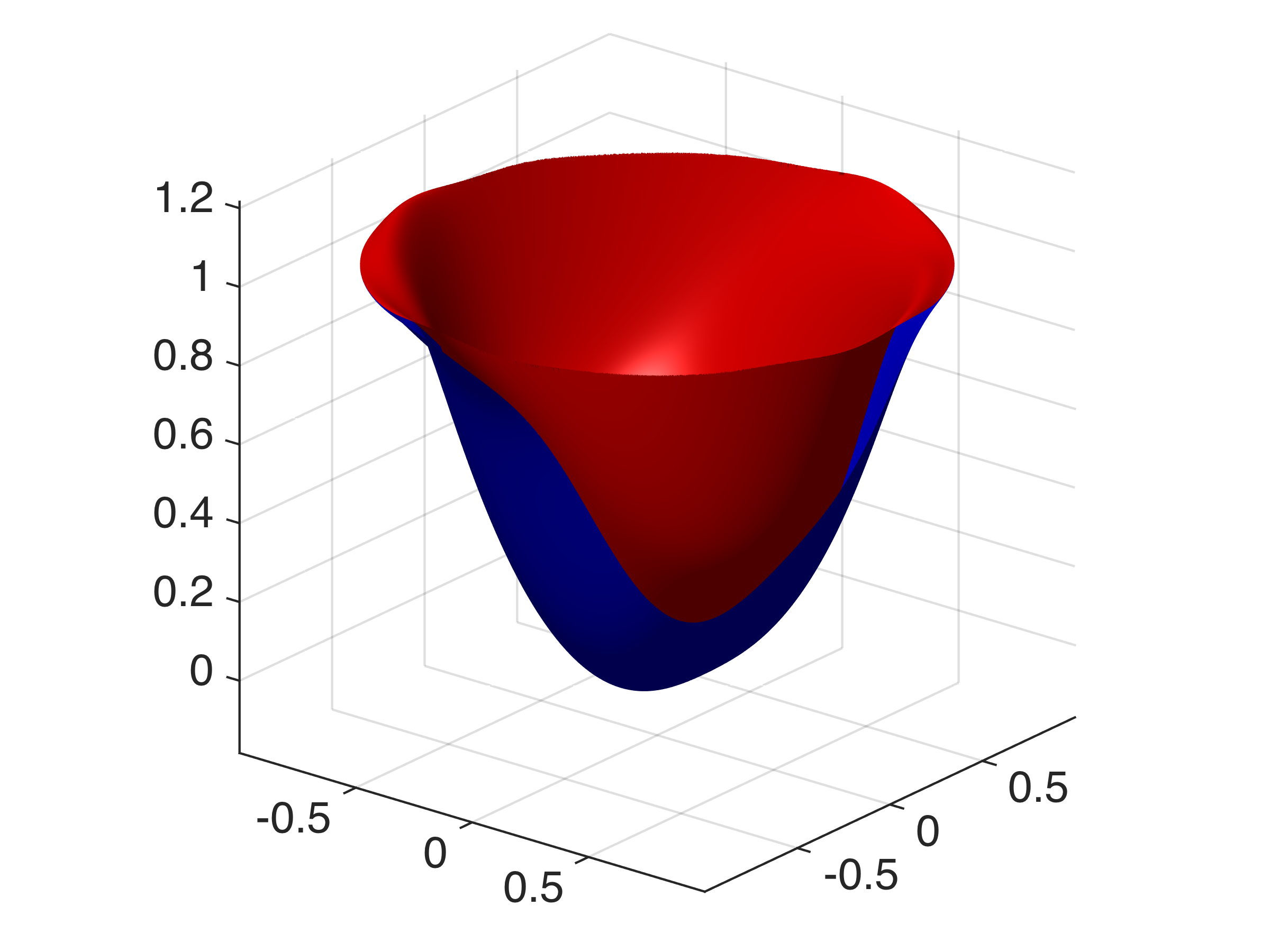}}
	\put(195,10){$x_1$}
	\put(312,14){$x_2$}
	
	\end{picture}
	\caption{\small Nonlinear double integrator -- upper bound on the value function $\overline{V_u}$ associated to the extracted controller (red); lower bound on the optimal value function $\underline{V}$ (blue).}
	\label{fig:NL_doubleInteg_valFun}
\end{figure*}


\subsection{Controlled Lotka-Volterra}
In our second example we apply the proposed method to a population model governed by $n$-dimensional controlled Lotka-Volterra equations
\[
\dot{x} = r \circ x \circ(\mathds{1} - Ax) + u^+ - u^{-},
\]
where $\mathds{1} \in \mathbb{R}^n$ is the vector of ones and $\circ$ denotes the componentwise (Hadamard) product. Each component $x_i$ of the state $x \in \mathbb{R}^n$ represents the size of the population of species $i$. The vector $r \in \mathbb{R}^n$ contains the intrinsic growth rates of each species and the matrix $A \in \mathbb{R}^{n\times n}$ captures the interaction between the species. If $A_{i,j} > 0$, then species $j$ is harmful to species $i$ (e.g., competes for resources) and if $A_{i,j} < 0$, then species $j$ is helpful to species $i$ (e.g., species $i$ feeds on species $j$); the diagonal components $A_{i,i}$ are normalized to one. The control inputs $u^+ \in [0,1]^n$ and $u^- \in [0,1]^n$ represent, respectively, the inflow and outflow of new species from the outside. For our numerical example we select $n = 4$ and model parameters
\[
r = \begin{bmatrix}    1 \\
    0.6 \\
    0.4\\
    0.2\end{bmatrix}, \quad A = \begin{bmatrix}    
    1 &   0.3 &   0.4 &   0.2 \\
   -0.2  &   1 &  0.4 &  -0.1 \\
   -0.1 &  -0.2  & 1   & 0.3 \\
   -0.1  & -0.2 & -0.3 &   1    
    \end{bmatrix},
\]
which results in a system with four states and eight control inputs. The economic objective is to harvest species number one while ensuring that no species goes extinct. More specifically the cost function is $l_u(x) = (-1.0, 0.5, 0.6, 0.8,    1.1, 2 ,4, 6)$ and $l_x(x) = 1$, where the vector $l_u(x)$ is associated with the control input vector $u = (u^-,u^+)$. Therefore there is a reward for harvesting species number one and cost associated with both introduction and hunting of all other species, the cost of hunting being lower than the cost of introduction. The reason for choosing $l_x(x) = 1$ is in order to make the joint stage cost $l(x,u)$~(\ref{eq:joinStageCost}) nonnegative; this choice does not affect optimality since $l_x(x(t)) = 1$ irrespective of the control input applied. The non-extinction constraint is expressed as $g(x) = 1- (Q^{-1}x - q)^\top (Q^{-1}x - q) \ge 0$ with $Q = \mathrm{diag}(0.475\cdot\mathds{1})$ and $q = 0.525\cdot\mathds{1}$. We choose $\beta = 1$ and $M = 16.16 > \sup_{x\in X, u\in U}\{l(u,x)\}/\beta  = 16$. We apply the coordinate transformation $x = Q\hat{x} + q$ and solve solve obtain a rational controller of degree eight by solving~(\ref{opt:dual_sos}). Figure~\ref{fig:lotkaVolterra} we shows plots for two different initial conditions, one with low population size of the first species and one with high. Finally, we evaluate the suboptimality of the extracted controller using the polynomial lower bound on the optimal value function of degree 11 obtained from~(\ref{opt:val_opt_lb}). Using Monte Carlo simulation with 1000 samples of initial conditions drawn from a uniform distribution over the constraint set we obtain average cost of the extracted controller to be 0.89 whereas the lower bounds predicts average cost of 0.72; hence the extracted controller is no more than $23.6\,\%$ suboptimal (modulo the statistical estimation  error). Note that we could also obtain a deterministic suboptimality estimate using the upper bound on the value function of the extracted controller obtained from~(\ref{opt:val_u_ub}). In this case, however, the upper bound~(\ref{opt:val_u_ub}) is not informative. Nevertheless, the Monte Carlo simulation along with the lower bound~(\ref{opt:val_opt_lb}) is a viable alternative in this case, since the extracted controller is simple and hence trajectories of the controlled system can be simulated rapidly.

\begin{figure*}[th]
\begin{picture}(140,180)
\put(0,90){\includegraphics[width=40mm]{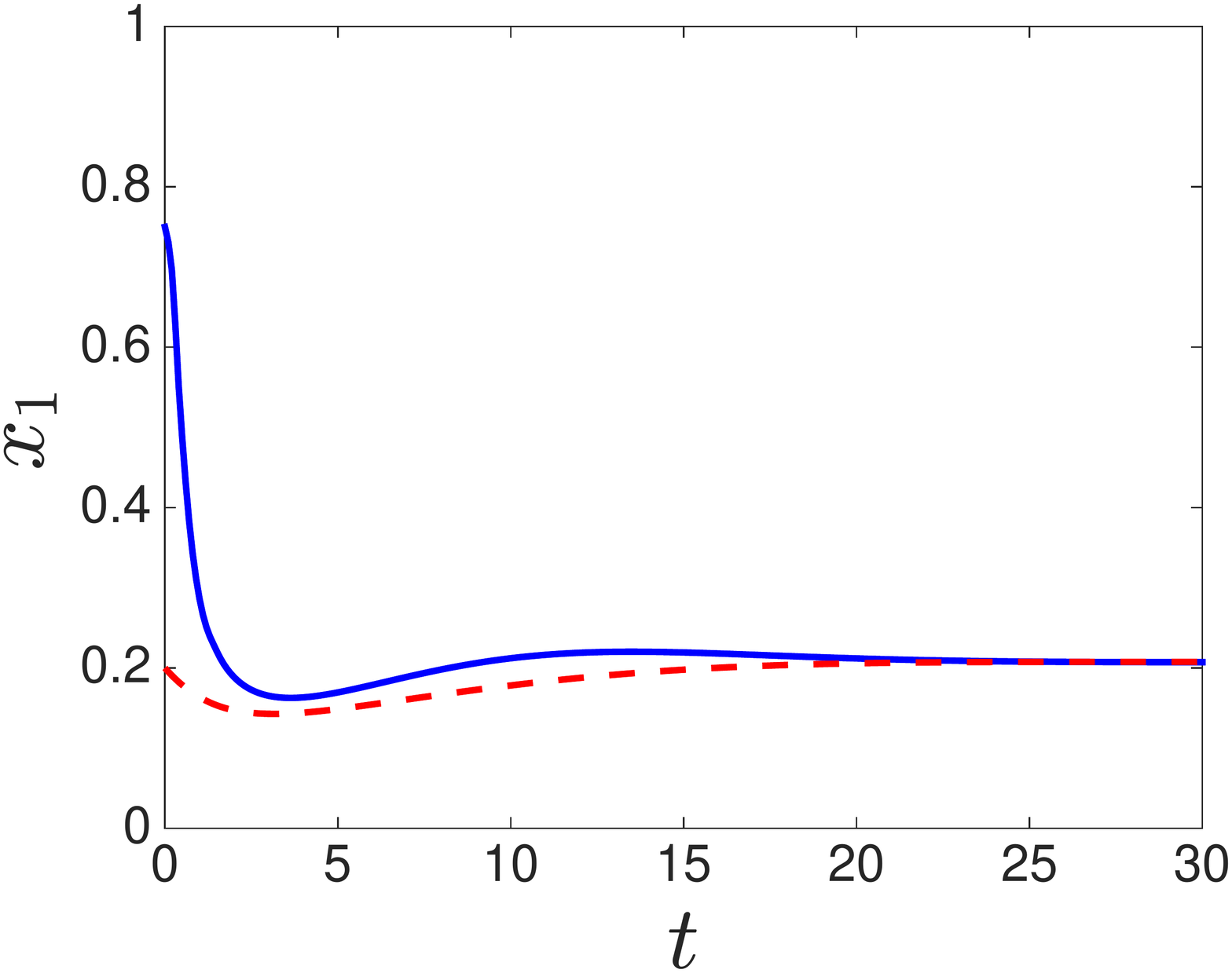}} 
\put(120,90){\includegraphics[width=40mm]{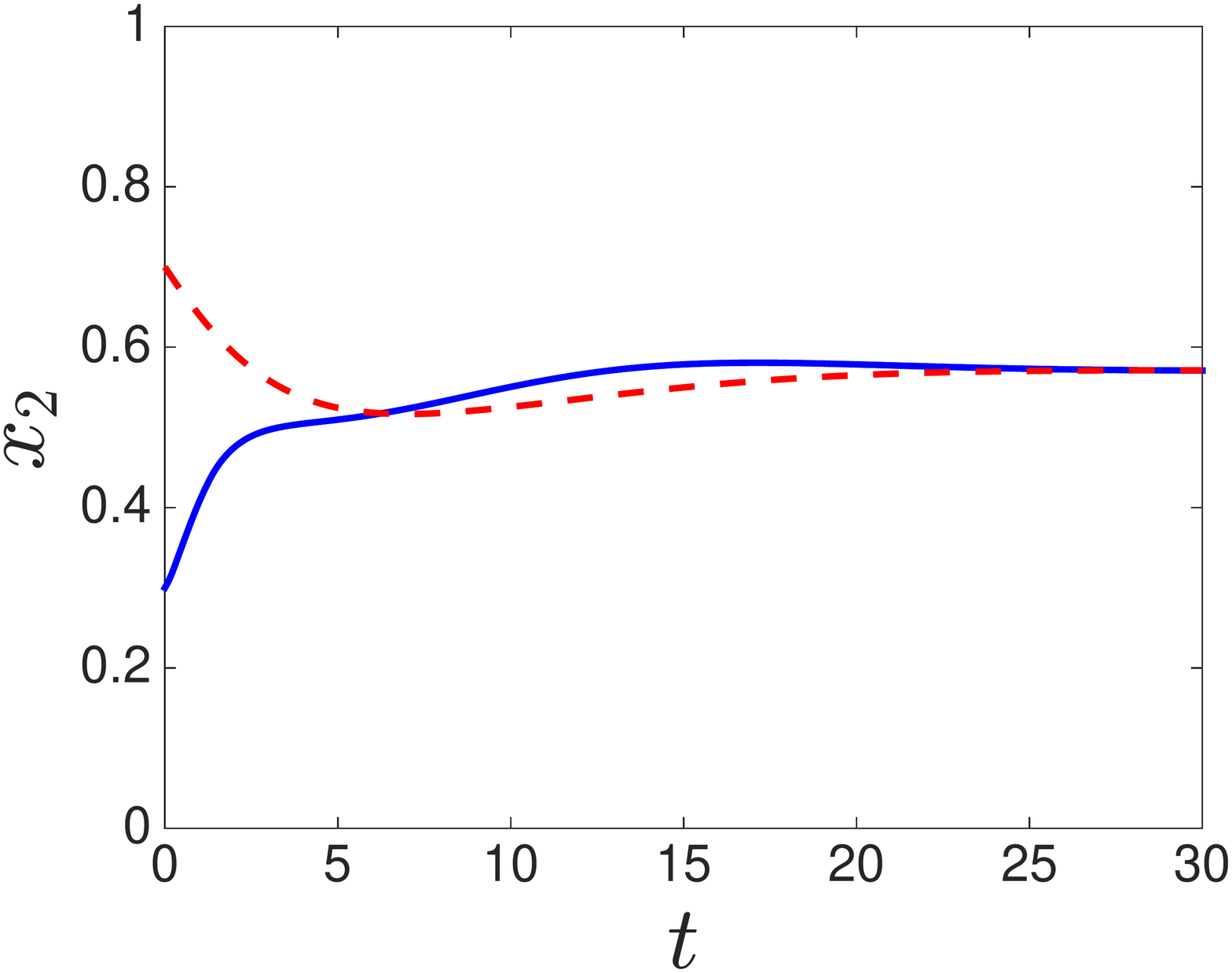}}
\put(240,90){\includegraphics[width=40mm]{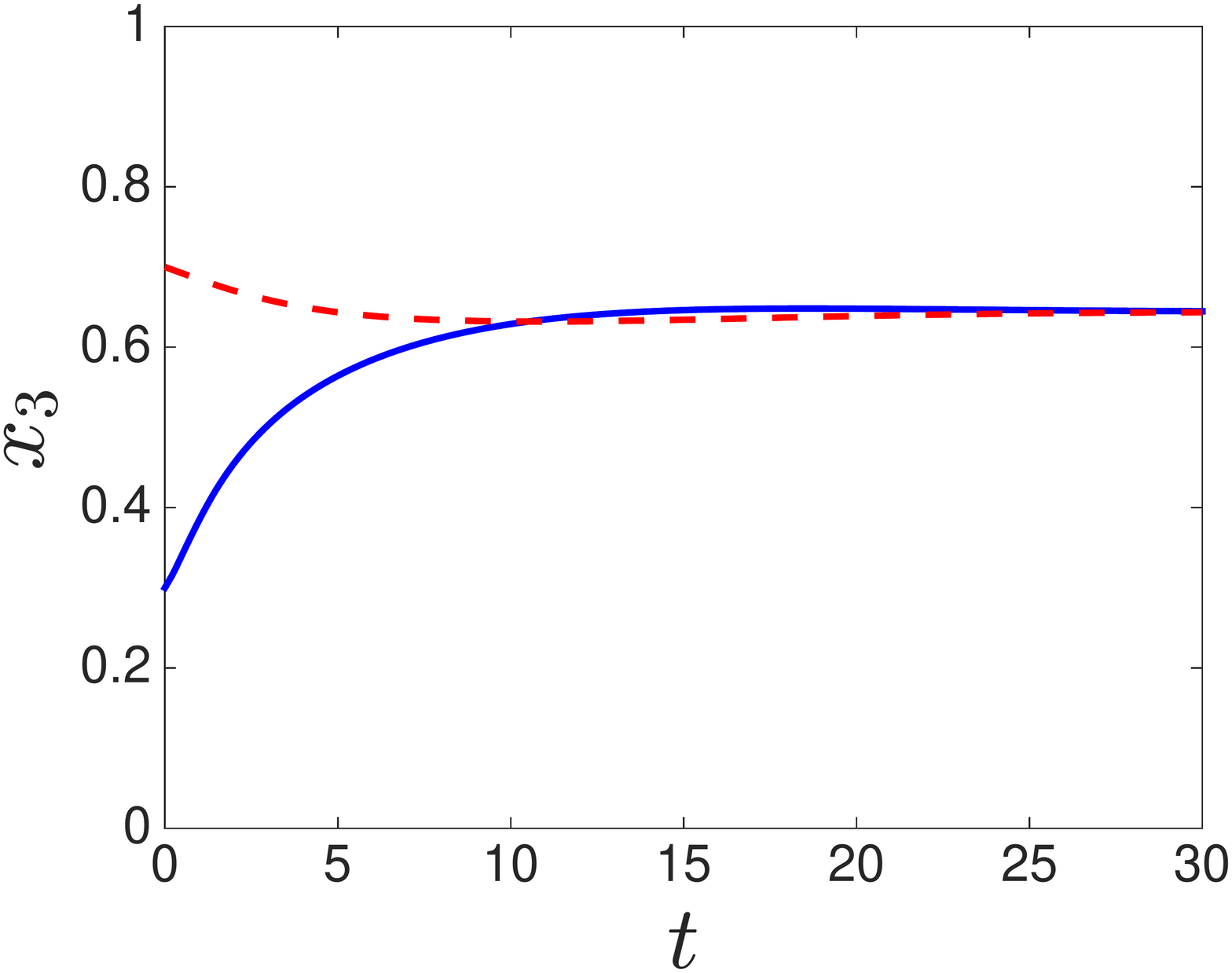}}
\put(360,90){\includegraphics[width=40mm]{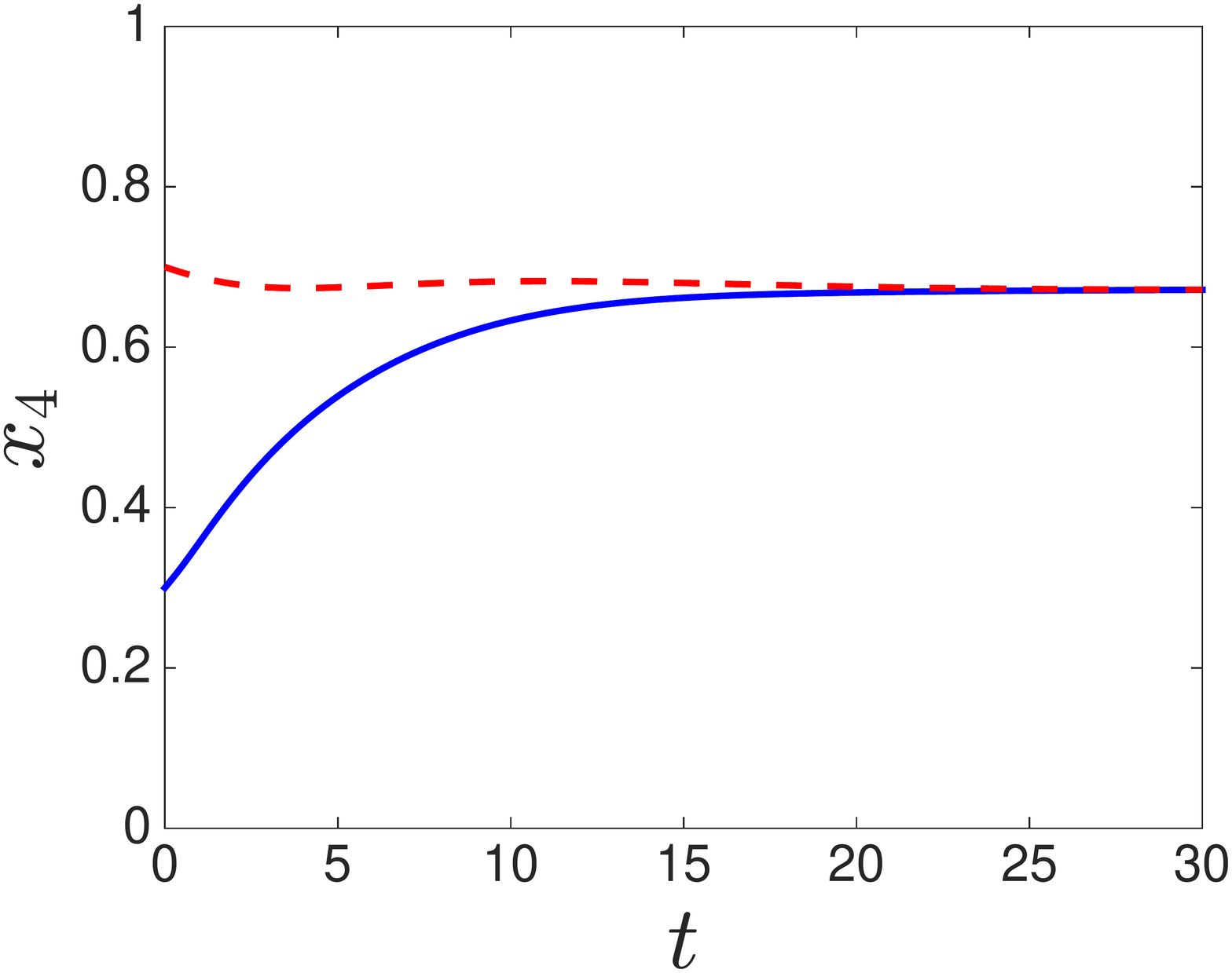}}

\put(0,0){\includegraphics[width=40mm]{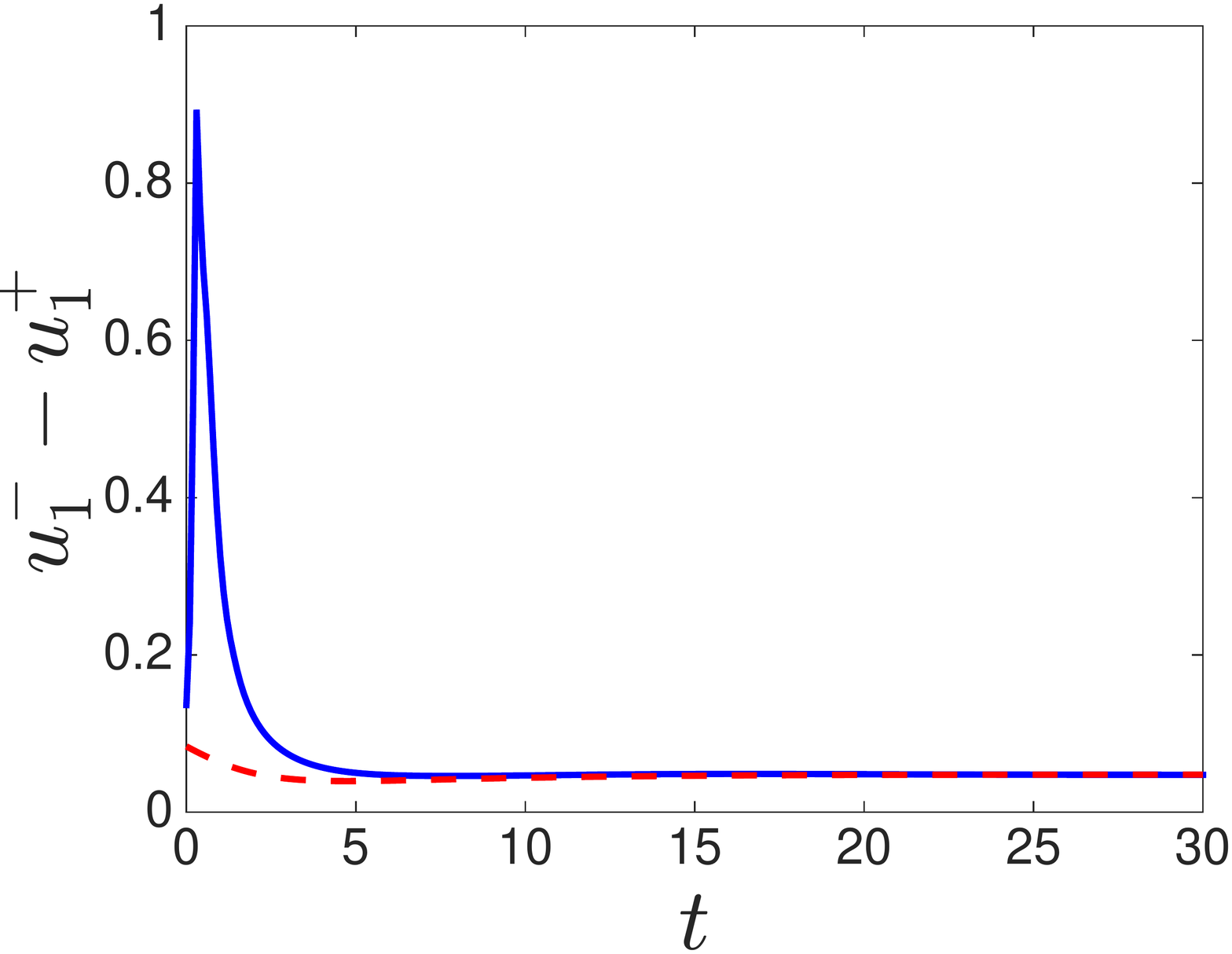}} 
\put(120,0){\includegraphics[width=40mm]{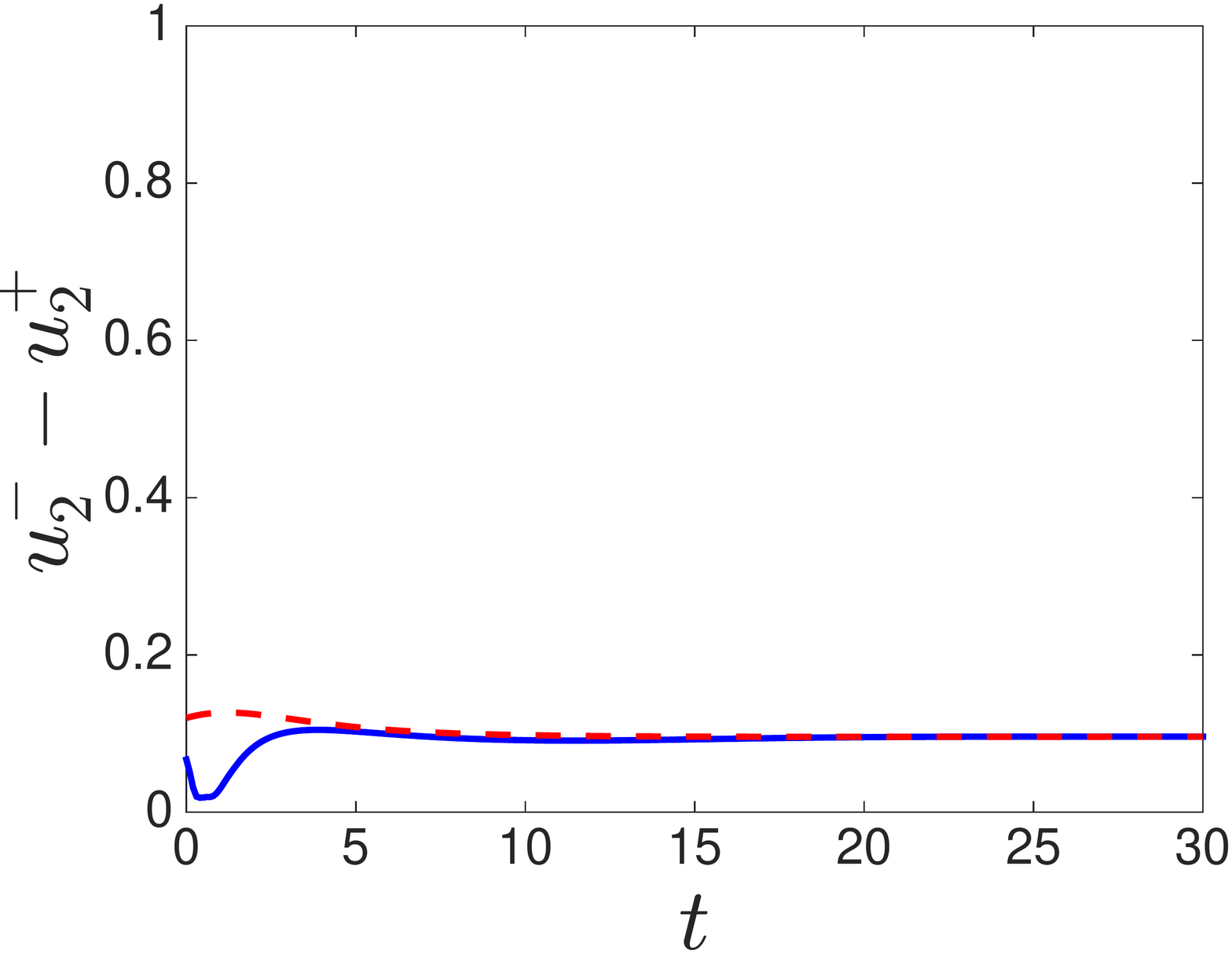}}
\put(240,0){\includegraphics[width=40mm]{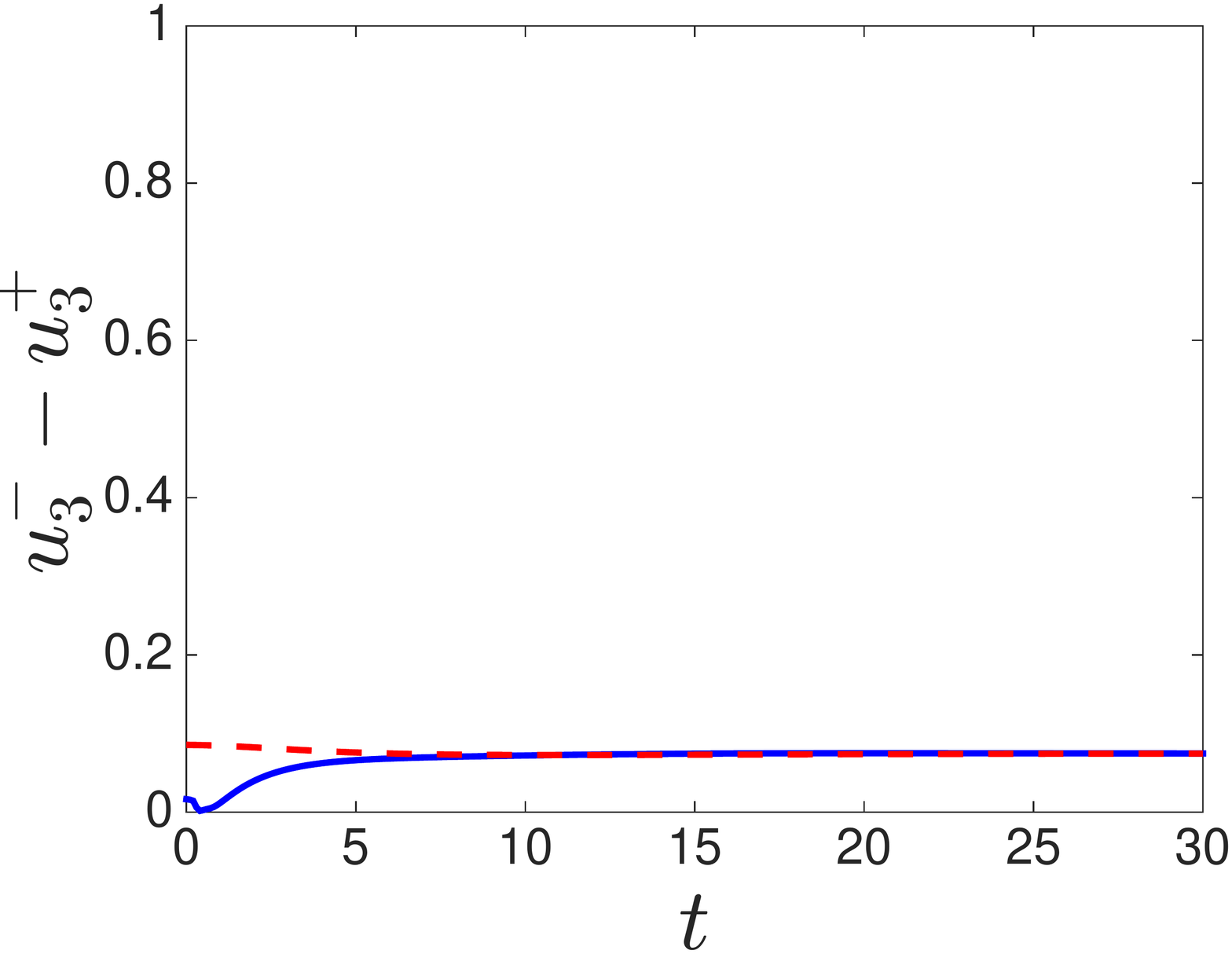}}
\put(360,0){\includegraphics[width=40mm]{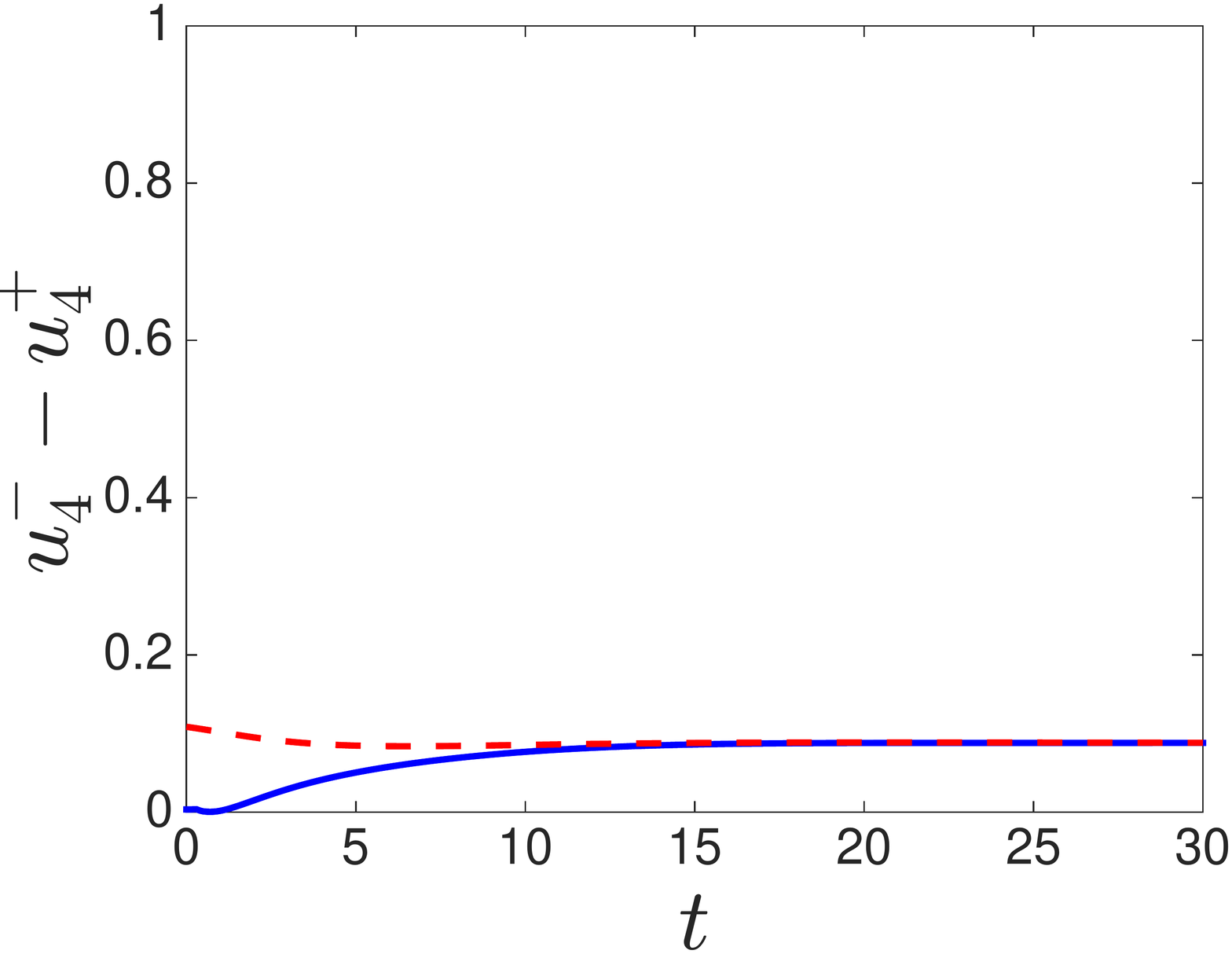}}



\end{picture}
\caption{Controlled Lotka-Volterra -- (blue) trajectory starting from a high initial population of the first species and low initial population of the other species; (red) trajectory starting from low initial population of the first species and high initial population of the other species.}
\label{fig:lotkaVolterra}
\end{figure*}

\section{Conclusion}
This paper presented a method to obtain a sequence of rational controllers asymptotically optimal (under suitable technical assumptions) in a discounted optimal control problem and a method to explicitly estimate suboptimality of each controller. The rational controller of a given degree is obtained by solving a single sum-of-squares problem with no extraction step. The SOS problem solved is feasible for any degree and therefore this method allows to trade off complexity of the controller against performance.

The approach is based on lifting the nonconvex optimal control problem into an infinite dimensional space of measures with continuous densities, where this problem becomes linear. Crucially, this problem is a tightening of the original problem, which follows immediately from the representation result for solutions of the discounted Liouville's equation with a terminal measure~(Theorem~\ref{thm:meas_rep_crucial_thm}). Asymptotic optimality of the extracted controllers then follows by approximating the asymptotically optimal continuous densities (guaranteed to exist by Assumption~\ref{as:smoothCont}) with polynomial densities in such a way that these densities correspond to the densities of the dynamical system (this is the essence of the proof of Theorem~\ref{thm:main}).

\section{Appendix A}
This Appendix contains the proof of Theorem~\ref{thm:Vu_ub_lb}; we use the same notation as in Section~\ref{sec:valApprox}. The inequalities $\overline{V_u}^d \ge V_u \ge \underline{V_u}^d$ follow from Gronwall's Lemma by noticing that the constraints of~(\ref{opt:val_u_ub}) and (\ref{opt:val_u_lb}) imply that 
\begin{equation}\label{eq:ubGron}
{\nabla}\overline{V_u}^{d} \cdot (f + \sum_{i=1}^m f_{u_i}u_i) \le \beta \overline{V_u}^{d} - (l_x + \sum_{i=1}^m l_{u_i}u_i) ,
\end{equation}
\begin{equation}
{\nabla}\underline{V_u}^{d} \cdot (f + \sum_{i=1}^m f_{u_i} u_i) \ge \beta \underline{V_u}^{d} - (l_x + \sum_{i=1}^m l_{u_i}u_i) 
\end{equation}
on $X$ and $\overline{V_u}  \ge  M$, $\underline{V_u} \le M$ on $\partial X$. We detail the argument for the inequality $\overline{V_u}^{d} \ge V_u$, the inequality $V_u \ge \underline{V_u}^d$ being similar. Given $x_0 \in X$ the inequality~(\ref{eq:ubGron}) implies that 
\[
\frac{d}{dt}\overline{V_u}^{d}(x(t\! \mid \! x_0)) \le \beta \overline{V_u}^{d}(x(t\! \mid \! x_0)) - \Big [ l_x(x(t\! \mid \! x_0)) + \sum_{i=1}^m l_{u_i}(x(t\! \mid \! x_0))u_i(x(t\! \mid \! x_0)) \Big ],
\] 
and therefore by Gronwall's Lemma
\[
\overline{V_u}^{d}(x(t\! \mid \! x_0)) \le e^{\beta t} \overline{V_u}^{d}(x_0) - \int_0^t e^{\beta(t-s)}\Big [ l_x(x(s\! \mid \! x_0)) + \sum_{i=1}^m l_{u_i}(x(s\! \mid \! x_0))u_i(x(s\! \mid \! x_0)) \Big ]\, ds
\]
and hence
\begin{equation}\label{eq:GrownAux}
\overline{V_u}^{d}(x_0) \ge e^{-\beta t}\overline{V_u}^{d}(x(t\! \mid \! x_0))  + \int_0^t e^{-\beta s}\Big [ l_x(x(s\! \mid \! x_0)) + \sum_{i=1}^m l_{u_i}(x(s\! \mid \! x_0))u_i(x(s\! \mid \! x_0)) \Big ]\, ds
\end{equation}
for all $t \in [0,\tau]$, where $\tau := \inf\{t\ge 0 \mid x(t\! \mid \! x_0) \notin X \} \in [0,\infty]$ is the first exit time of $X$. Next we observe that $V_u$, the value function associated to $u$, is equal to
\[
V_u(x_0) =\begin{cases}
 \int_0^\infty e^{-\beta s}\Big [ l_x(x(s\! \mid \! x_0)) + \sum_{i=1}^m l_{u_i}(x(s\! \mid \! x_0))u_i(x(s\! \mid \! x_0)) \Big ]ds, & \tau = \infty \\
 Me^{-\beta \tau} + \int_0^\tau e^{-\beta s}\Big [ l_x(x(s\! \mid \! x_0)) + \sum_{i=1}^m l_{u_i}(x(s\! \mid \! x_0))u_i(x(s\! \mid \! x_0)) \Big ]ds, & \tau < \infty.
\end{cases}
\]
In view of~(\ref{eq:GrownAux}), we conclude that $\overline{V_u}^{d}(x_0) \ge V_u(x_0)$ if $\tau = \infty$ since $\overline{V_u}^{d}$ is polynomial and hence bounded on $X$ (and hence $e^{-\beta t}\overline{V_u}^{d}(x(t\! \mid \! x_0)) \to 0$); and we conclude that $\overline{V_u}^{d}(x_0) \ge V_u(x_0)$ if $\tau < \infty$ since $x(\tau \! \mid \!  x_0) \in \partial X$ and $\overline{V_u}^{d} \ge M$ on $\partial X$.

Convergence of the upper and lower bounds~(\ref{eq:convUbLb}) follows from Theorem~\ref{thm:meas_rep_crucial_thm} using infinite-dimensional LP duality and standard results on the convergence of moment relaxations. The proof is similar to the proof of Theorem 5 in~\cite{mci_outer} or Theorem 3.6 in~\cite{sicon} and therefore we only outline it. The hierarchy of SOS programming problems~(\ref{opt:val_u_ub}) and (\ref{opt:val_u_lb}) is dual to the hierarchy of moment relaxations of an infinite-dimensional LP in the cone of nonnegative measures whose dual is an infinite-dimensional LP in $C^1(X)$ and feasible solutions of this dual provide upper or lower bounds on $V_u$. Crucial to applying infinite-dimensional duality results (e.g., \cite[Theorem 3.10]{anderson}) is the boundedness of measures satisfying the discounted Liouville equation~(\ref{eq:discountLiouville_general}) with $\nu_i \le \bar{u}\mu$ and  $\mu_0 = \lambda_X$, where $\lambda_X$ is the restriction of the Lebesgue measure to $X$. Plugging $v = 1$ in~(\ref{eq:discountLiouville_general}) we have $\mu_T(X) + \beta \mu(X) = \mu_0(X)$. Since $ \mu_0(X) = \lambda_X(X)=  \mathrm{vol}\:X < \infty$ and $\beta > 0$ we conclude that $\mu_T$ and $\mu$ are indeed bounded, which implies that $\nu_i$ is also bounded for $i=1,\ldots,m$. Equally important is the absence of duality gap between the finite-dimensional moment relaxations and SOS tightenings (which are both SDP problems); this follows immediately from the presence of the constraint $g_i = N - \|x \|_2^2$ among the constraints describing $X$, which implies the boundedness of the truncated moment sequences feasible in the moment relaxations. The absence of duality gap then follows from~\cite[Lemma~2]{trnovska}. \hfill $\square$

\section{Appendix B}
This appendix presents a proof of Theorem~\ref{thm:meas_rep_crucial}. We will prove a slightly more general version of the result from which Theorem~\ref{thm:meas_rep_crucial} immediately follows:
\begin{theorem}\label{thm:meas_rep_crucial_thm}
Let $\bar{f} : \mathbb{R}^n\to\mathbb{R}^n$ be globally Lipschitz and let the nonnegative measures $\mu$, $\mu_0$, $\mu_T$ 
on $\mathbb{R}^n$ satisfy
\begin{equation}\label{eq:discountLiouville_general_thm}
\int_{\mathbb{R}^n}  v \,d\mu_T =
 \int_{\mathbb{R}^n} v \,d\mu_0  + \int_{\mathbb{R}^n} (\nabla v \cdot \bar{f} - \beta v) \, d\mu
\end{equation}
for all $v \in C^1(\mathbb{R}^n)$. Then there exists an ensemble of probability measures $\{ \tau_{x_0} \}_{x_0 \in X }$
with $\mathrm{spt}\,\tau_{x_0} \subset [0,\infty]$ and an ensemble of trajectories  $\{ x(\cdot \!\mid\! x_0)\}_{x_0\in X}$
of the ODE~$\dot{x} = \bar{f}(x)$ and
\begin{subequations}\label{eq:meas_traj_thm}
\begin{align}
\int_{\mathbb{R}^n} v(x) \,d\mu_0(x) &= \int_{\mathbb{R}^n} v(x(0\!\mid\! x_0)) \,d\mu_0(x_0),\\
\int_{\mathbb{R}^n} v(x)\,d\mu(x) &= \int_{\mathbb{R}^n}  \int_0^{\infty} \int_0^{\tau} e^{-\beta t}v(x(t\!\mid\!x_0))\,dt\, d\tau_{x_0}(\tau) \,d\mu_0(x_0), \label{eq:meas_mu_thm} \\
  \int_{\mathbb{R}^n} v(x)\,d\mu_T(x) &= \int_{\mathbb{R}^n} \int_0^{\infty} e^{-\beta \tau }v(\tau(x_0)) \, d\tau_{x_0}(\tau) \,d\mu_0(x_0), \label{eq:meas_muT_thm}
\end{align}
\end{subequations}
for all $v\in L^1(\mathbb{R}^n)$.
\end{theorem}
Theorem~\ref{thm:meas_rep_crucial} follows from Theorem~\ref{thm:meas_rep_crucial_thm} by setting $\bar{f}= f + \sum_{i=1}^m f_{u_i}u_i$ and modifying $f$ and $f_{u_i}$ outside the compact set $X$ such that $\bar{f}$ is globally Lipschitz\footnote{Such modification is always possible. For instance let $\bar{f}(x) = \min_{y \in X}\{f(y) + \sum_{i=1}^m f_{u_i}(y)u_i(y) + L\|x-y\| \}$, where $L$ is the Lipschitz constant of $f + \sum_{i=1}^m f_{u_i}u_i$ on $X$.}. The conclusion that $x(t\!\mid \! x_0) \in X$ for all $t \in \mathrm{spt}\,\tau_{x_0}$ follows by taking $v(x) = e^{- \|x \|^2}I_{\mathbb{R}^n \setminus X}(x)$ in~(\ref{eq:meas_traj_thm}), where $I_A$ is the indicator function of a set $A$, i.e. $I_A(x)= 1$ if $x \in A$ and $I_A(x) = 0$ otherwise.

Suppose therefore that~(\ref{eq:meas_traj_thm}) holds. First we will prove a simple result. In the rest of this Appendix we will use the notation $C_c^k$ for the space of all compactly supported $k$-times continuously differentiable functions.
\begin{lemma}\label{lem:transPortEq_sol}
For any $w \in C_c^1(\mathbb{R}^n)$, the equation
\begin{equation}\label{eq:trans_disc}
\nabla v \cdot \bar f - \beta v = w
\end{equation}
has a solution $v$ such that for all $x_0 \in \mathbb{R}^n$ it holds
\begin{equation}\label{eq:v_rep}
v(x_0) = -\int_0^\infty e^{-\beta t} w(x(t \! \mid \! x_0))\, dt.
\end{equation}
\end{lemma}

\begin{proof}
Since $\bar{f}$ is globally Lipschitz the solution $x(t\! \mid \! x_0)$ is defined for all $x_0 \in \mathbb{R}^n$ and all $t \ge 0$. Therefore~(\ref{eq:v_rep}) is well defined (notice that $w$ is bounded and $\beta > 0$). Direct computation then gives:
\begin{align*}
\nabla v \cdot \bar{f} (x(t\mid x_0)) & = \frac{d}{dt}v(x(t \! \mid \! x_0)) \\
& = -  \frac{d}{dt} \int_0^\infty e^{-\beta s} w(x(s \! \mid \! x(t \!\mid \! x_0) )) \, ds \\
& = -  \frac{d}{dt} \int_0^\infty e^{-\beta s} w(x(t+s \! \mid \! x_0)) \, ds \\ 
& = -  \int_0^\infty e^{-\beta s} {\nabla}w(x(t+s \! \mid \! x_0)) \cdot \bar{f}(x(t+s \! \mid \! x_0)) \, ds \\ 
& = -  \int_0^\infty e^{-\beta s} {\nabla}w(t+s \! \mid \! x_0)) \cdot \bar{f}(x(t+s \! \mid \! x_0)) \, ds \\
& = -  \int_0^\infty e^{-\beta s} \frac{d}{ds}w(x(t+s \! \mid \! x_0)) \, ds \\ 
& = -  \beta \int_0^\infty e^{-\beta s}w(x(t+s \! \mid \! x_0)) \, ds - [e^{-\beta s}w(x(t+s \! \mid \! x_0)) ]_{0}^\infty \\ 
& = -  \beta \int_0^\infty e^{-\beta s}w(x(s \! \mid \! x(t \!\mid \! x_0) )) \, ds  + w(x(t \!\mid \! x_0) ) \\ 
& = \beta v(x(t \! \mid \! x_0)) + w(x(t \!\mid \! x_0) ).
\end{align*}
Setting $t = 0$, we arrive at~(\ref{eq:trans_disc}).
\end{proof}

\emph{Proof (of Theorem~\ref{thm:meas_rep_crucial_thm})} We will proceed in several steps.
\paragraph{Two Diracs.} We start with the simplest case of $\mu_0 = \delta_{x_0}$ and $\mu_T = a\delta_{x_T}$, $a > 0$, $x_T \in\mathbb{R}^n$, and some $\mu \ge 0$. First we will show that if $(\mu_T,\mu_0,\mu)$ solves~(\ref{eq:discountLiouville_general_thm}) then there exists a time $\tau \ge 0$ such that $x(\tau \!\mid \! x_0) = x_T$. Consider now any $w\in C_c^1(\mathbb{R}^n)$, $w \ge 0$ and the associated $v\in C^1(\mathbb{R}^n)$ solving~(\ref{eq:trans_disc}). Then we have
\[
av(x_T) - v(x_0) = \int_{\mathbb{R}^n}( \nabla v \cdot f - \beta v)\, d\mu =  \int_{\mathbb{R}^n} w\, d\mu \ge 0.
\]
Therefore, by Lemma~\ref{lem:transPortEq_sol},
\[
a v(x_T) \ge v(x_0) = -\int_0^\infty e^{-\beta t} w(x(t \! \mid \! x_0))\, dt.
\]
Using~(\ref{eq:v_rep}) again on $v(x_T)$ we get
\[
 -a\int_0^\infty e^{-\beta t} w(x(t \! \mid \! x_T))\, dt \ge  -\int_0^\infty e^{-\beta t} w(x(t \! \mid \! x_0))\, dt,
\]
or
\begin{equation}\label{eq:derivAux}
 a\int_0^\infty e^{-\beta t} w(x(t \! \mid \! x_T))\, dt \le  \int_0^\infty e^{-\beta t} w(x(t \! \mid \! x_0))\, dt.
\end{equation}
Now pick $S \ge 0$ (to be specified later) and consider the traces
\[
\mathcal{X}_0 = \{ x(t\mid x_0) \mid 0\le t \le S\}.
\]
\[
\mathcal{X}_T = \{ x(t\mid x_T) \mid 0\le t \le S\}.
\]
Assuming there is no $\tau \ge 0$ such that $x(\tau \!\mid \! x_0) = x_T$ we have $\mathcal{X}_0\cap \mathcal{X}_T = \emptyset$ and since $\mathcal{X}_0$ and $\mathcal{X}_T$ are compact there exist (by Uryshon's Lemma with mollification) a function $w \in C_c^1(\mathbb{R}^d ; [0,1])$ such that $w= 0$ on $\mathcal{X}_0$ and $w = 1$ on~$\mathcal{X}_T$.
Then the left hand side of~(\ref{eq:derivAux}) is greater than or equal to $a (1-e^{-\beta S}) / \beta$
whereas the right hand side is less than or equal to $e^{-\beta S} / \beta$.
Since $a > 0$ and $\beta > 0$ we arrive at a contradiction with~(\ref{eq:derivAux}) by picking a sufficiently large $S$. Therefore there exists a $\tau \ge 0$ such that $x(\tau \!\mid\! x_0) = x_T$ (i.e., $x_T$ and $x_0$ are on the same trace of the flow associated to $\dot{x} = f(x)$).

Now we prove that $a \le e^{-\beta \tau}$. Since $x_T = x(\tau)$ and $x_0$ are on the same trace we have
\[
v(x_0) = e^{-\beta \tau}\underbrace{v(x_T)}_{v(x(\tau))} - \int_0^\tau w(x(t\!\mid\! x_0)) \, dt.
\]
Using again $av(x_T) \ge v(x_0)$ if $w \ge 0$ we get
\[
a v(x_T) \ge e^{-\beta \tau}v(x_T) - \int_0^\tau w(x(t\!\mid\! x_0)) \, dt, \:\:\mathrm{or}
\]
\begin{equation}\label{eq:asmall}
(e^{-\beta \tau }-a) \int_0^\infty e^{-\beta t} w(x(  t  \mid\! x_T   )) dt \ge - \int_0^\tau w(x(t\!\mid\! x_0)) \, dt.
\end{equation}
Consider the set
\[
\mathcal{X}_\tau = \{ x(t\mid x_0) \mid 0\le t \le \tau\}.
\]
Since $x_0$ and $x_T$ are on the same trace (and $x_T$ follows $x_0$) there exists $w \in C_c^1(X)$, $w\ge 0$, such that $w= 0$ on $\mathcal{X}_\tau$ and $w > 0$ elsewhere (e.g., let $w(x) = \min(\mathrm{dist}(x, \mathcal{X}_\tau),1)$ with appropriate mollification). With this choice of $w$ the equation~(\ref{eq:asmall}) gives
\[
(e^{-\beta \tau }-a) \int_0^\infty e^{-\beta t} w(x(  t  \mid\! x_T   )) dt \ge 0
\]
and therefore $a \le e^{-\beta \tau}$ since the integral is strictly positive. This proves the first two claims.

To finish we observe that any solution to~(\ref{eq:trans_disc}) satisfies
\[
e^{-\beta \tau }v(x_T) = v(x_0) + \int_0^\tau e^{-\beta t} w(x(t \! \mid \! x_0))\, dt.
\]
Therefore
\[
av(x_T) = v(x_0)ae^{\beta \tau} + ae^{\beta \tau}\int_0^\tau e^{-\beta t} w(x(t \! \mid \! x_0))\, dt.
\]
Using~(\ref{eq:v_rep}) we get
\[
av(x_T) = v(x_0) + \underbrace{ae^{\beta \tau}}_{\ge 0}\int_0^\tau e^{-\beta t} w(x(t \! \mid \! x_0))\, dt +\underbrace{ (1-ae^{\beta \tau})}_{\ge 0}\int_0^\infty e^{-\beta t} w(x(t \! \mid \! x_0))\, dt.
\]
Since
\[
av(x_T) - v(x_0) = \int_{\mathbb{R}^n} w \, d\mu
\]
we conclude that
\[
\int_{\mathbb{R}^n} w\, d\mu = ae^{\beta \tau}\int_0^\tau e^{-\beta t} w(x(t \! \mid \! x_0))\, dt + (1-ae^{\beta \tau})\int_0^\infty e^{-\beta t} w(x(t \! \mid \! x_0))\, dt,
\]
i.e., $\mu$ is indeed generated by trajectories of $\dot{x} = f(x)$ (in this case by two trajectories, both starting at $x_0$, one stopping at $\tau$, the other one continuing to infinity with weights given by the ratio of masses of $\mu_0$ and $\mu_T$). That is the measure $\tau_{x_0}$ is given by
\[
\tau_{x_0} = ae^{\beta \tau}\delta_\tau  + (1-ae^{\beta \tau})\delta_\infty
\]
as expected.

\paragraph{Dirac at $x_0$, sum of Diracs for $\mu_T$.}
Next we treat the case where $\mu_T = \sum_{i=1}^\infty a_i\delta_{x_i}$ for some $a_i \ge 0$ and $x_i \in \mathbb{R}^n$. Using the same argument as in the previous case we can show that
\[
x_i \in \mathcal{X}_0 = \{x(t\mid x_0) \mid t\ge 0\} 
\]
for all $i$ and that the condition
\[
\sum_{i=1}^\infty a_i e^{\beta \tau _i} \le 1,
\]
holds with $\tau_i$ being the times to reach $x_i$ from $x_0$. Then we have
\[ \tau_{x_0} = \sum_{i=1}^\infty a_i e^{\beta \tau_i} \delta_{\tau_i} + (1- \sum_{i=1}^\infty a_i  e^{\beta \tau_i}) \delta_{\infty}.
\]

\paragraph{Dirac at $x_0$ arbitrary $\mu_T$.}
In the same way as before we can show that the support of $\mu_T$ must be on the trace $\mathcal{X}_0$. Then we can define the measure $\hat\tau_{x_0} $ by
\[
\hat\tau_{x_0} (A) := \mu_T(x(A \mid x_0)),\quad A \subset [0,\infty)
\]
and show that it has to satisfy the condition $\int_0^{\infty} e^{\beta t } d\hat{\tau}_{x_0}(t) \le 1$.
Next, using the fact that the mapping $t \mapsto x(t \! \mid \! x_0)$ is invertible, we obtain
\[
\int_{\mathbb{R}^n} v \, d\mu_T = \int_0^{\infty}  v( x(t \! \mid \! x_0)  ) \, d\hat{\tau}_{x_0}(t).
\]
The conclusion of the theorem then holds with $\tau_{x_0}$ defined by
\[
\tau_{x_0}(A) = \int_0^{\infty} I_A(t) e^{\beta t } d\hat{\tau}_{x_0}(t) + \Big [1- \int_0^{\infty} e^{\beta t } d\hat{\tau}_{x_0}(t) \Big ] I_A(\infty), \quad A\subset [0,\infty].
\]


\paragraph{Arbitrary $\mu_0$, arbitrary $\mu_T$.}
The general case follows by approximating $\mu_0$ by a sum of Dirac measures, using the fact that any measure is the weak limit of a sequence of Dirac measures.

\end{document}